\documentclass[final]{siamltex}

\usepackage{graphicx}
\usepackage{color}
\usepackage{bm}
\usepackage{amsfonts}
\usepackage{array}
\usepackage{amsmath}
\usepackage{amssymb}
\usepackage{latexsym}
\usepackage{enumerate}
\usepackage{subfigure}
\usepackage{amscd}

\newcommand{\lmb}{\mbox{\boldmath$\lambda$}}
\newcommand{\tr}{^{\sf T}}
\newcommand{\m}[1]{{\bf{#1}}}
\newcommand{\g}[1]{\bm #1}
\newcommand{\C}[1]{{\cal {#1}}}
\renewcommand{\bar}{\overline}
\newtheorem{remark}{Remark}[section]

\title{Convergence Rate for a Radau hp Collocation Method Applied to
Constrained Optimal Control
\thanks{
October 24, 2017, revised September 11, 2018.
The authors gratefully acknowledge support by
the Office of Naval Research under grants N00014-11-1-0068,
N00014-15-1-2048, and N00014-18-1-2100,
by the National Science Foundation under
grants DMS-1522629, CBET-1404767, and DMS-1819002,
and  by the U.S.~Air Force
Research Laboratory under contract FA8651-08-D-0108/0054.}
}
\author{
William W. Hager\thanks{{\tt hager@ufl.edu},
http://people.clas.ufl.edu/hager/,
PO Box 118105,
Department of Mathematics,
University of Florida, Gainesville, FL 32611-8105.
Phone (352) 294-2308. Fax (352) 392-8357.}
\and
Hongyan Hou\thanks{{\tt hongyan.hou@mnstate.edu},
        Mathematics Department,
        Minnesota State University Moorhead,
        P.O. Box 104, 1104 7th Avenue South, Moorhead, MN 56563
        Phone (218) 477-4007.}
\and
Subhashree Mohapatra\thanks{{\tt subha@ufl.edu},
        Department of Mathematics,
        University of Florida, Gainesville, FL 32611.}
\and
Anil V. Rao\thanks{{\tt anilvrao@ufl.edu},
        http://www.mae.ufl.edu/rao,
        Department of Mechanical and Aerospace Engineering,
        P.O. Box 116250, Gainesville, FL 32611-6250.
        Phone (352) 392-0961. Fax:(352) 392-7303.}
\and
Xiang-Sheng Wang\thanks{{\tt xswang@louisiana.edu},
        http://www.ucs.louisiana.edu/$\sim$xxw6637/,
        Department of Mathematics,
        University of Louisiana at Lafayette,
        Lafayette, LA 70503.
        Phone (337) 482-5281.}
}

\begin{document}
\maketitle
\begin{abstract}
For control problems with control constraints,
a local convergence rate is established for an $hp$-method
based on collocation at the Radau quadrature points in each mesh
interval of the discretization.
If the continuous problem has a sufficiently smooth solution
and the Hamiltonian satisfies a strong convexity condition,
then the discrete problem possesses a local minimizer
in a neighborhood of the continuous solution, and as either the number
of collocation points or the number of mesh intervals increase,
the discrete solution convergences to the continuous solution in the sup-norm.
The convergence is exponentially fast with respect to the degree of the
polynomials on each mesh interval, while the error is bounded by a
polynomial in the mesh spacing.
An advantage of the $hp$-scheme over global polynomials is that
there is a convergence guarantee when the mesh is sufficiently small,
while the convergence result for global polynomials requires that a
norm of the linearized dynamics is sufficiently small.
Numerical examples explore the convergence theory.
\end{abstract}



\pagestyle{myheadings} \thispagestyle{plain}
\markboth{W. W. HAGER, H. HOU, S. MOHAPATRA, A. V. RAO, AND X.-S. WANG}
{RADAU HP COLLOCATION FOR OPTIMAL CONTROL}

\section{Introduction}
\label{introduction}
A convergence rate is established for an $hp$-orthogonal collocation method
applied to a constrained control problem of the form
\begin{equation}\label{P}
\left. \begin{array}{cll}
\mbox {minimize} &C(\m{x}(1))&\\
\mbox {subject to} &\dot{\m{x}}(t)=
\m{f(x}(t), \m{u}(t)),& \m{u}(t) \in \C{U}, \quad t\in\Omega_0,\\
&\m{x}(0)=\m{a},&
(\m{x}, \m{u}) \in \C{C}^1(\Omega_0) \times
\C{C}^0 (\Omega_0),
\end{array} \right\}
\end{equation}
where $\Omega_0 = [0, 1]$, the control constraint set
$\C{U}\subset \mathbb{R}^m$ is closed and convex with nonempty interior,
the state ${\m x}(t)\in \mathbb{R}^n$,
$\dot{\m x}$ denotes the derivative of $\m{x}$ with respect to $t$,
${\m f}: {\mathbb R}^n \times {\mathbb R}^m\rightarrow {\mathbb R}^n$,
$C: {\mathbb R}^n \rightarrow {\mathbb R}$,
and ${\m a}$ is the initial condition, which we assume is given;
$\C{C}^l(\Omega_0)$
denotes the space of $l$ times continuously differentiable functions
mapping $\Omega_0$ to $\mathbb{R}^d$ for some $d$.
The value of $d$ should be clear from context; states and costates always
have $n$ components and controls have $m$ components.
It is assumed that $\m{f}$ and $C$ are at least continuous.
When the dynamics in (\ref{P}) can be solved for
the state $\m{x}$ as a function of the control $\m{u}$,
the control problem reduces to a constrained minimization over $\m{u}$.

The development of $hp$-techniques in the context of finite element
methods for boundary-value problems began with the work of
Babu\v{s}ka and Gui in \cite{Babuska1986a, Babuska1986b, Babuska1986c},
and Babu\v{s}ka and Suri in \cite{Babuska1987, Babuska1990, Babuska1994}.
In the $hp$-collocation approach that we develop for (\ref{P}),
the time domain $\Omega_0$ is initially partitioned into a mesh.
To simplify the discussion, we focus on a uniform mesh consisting
of $K$ intervals $[t_{k-1}, t_{k}]$ defined by the mesh points
$t_k = k/K$ where $0 \le k \le K$.
The dynamics of (\ref{P}) are reformulated using a change of variables.
Let $t_{k+1/2} = (t_k + t_{k+1})/2$ be the midpoint of the mesh
interval $[t_k, t_{k+1}]$.
We make the change of variables $t = t_{k-1/2} + h\tau$,
where $h = 1/(2K)$ is half the width of the mesh interval and
$\tau \in \Omega := [-1, 1]$;
let us define $\m{x}_k : \Omega \rightarrow \mathbb{R}^n$ by
$\m{x}_k (\tau) = \m{x}(t_{k-1/2} + h\tau)$.
Thus $\m{x}_k$ corresponds to the restriction of $\m{x}$ to the mesh
interval $[t_{k-1}, t_{k}]$.
Similarly, we define a control $\m{u}_k$ corresponding to the
restriction of $\m{u}$ to the mesh interval $[t_{k-1}, t_{k}]$.
In the new variables, the control problem reduces to finding
$K$ state-control pairs $(\m{x}_k, \m{u}_k)$, $1 \le k \le K$,
each pair defined on the interval $[-1, 1]$, to solve the problem
\begin{equation}\label{Dtau}
\left. \begin{array}{cll}
\mbox {minimize} &C(\m{x}_K(1))&\\
\mbox {subject to} &\dot{\m{x}}_k(\tau)= h
\m{f(x}_k(\tau), \m{u}_k(\tau)),\quad \m{u}_k(\tau) \in \C{U}, &\tau\in\Omega,\\
&\m{x}_k(-1)=\m{x}_{k-1}(1),& 1 \le k \le K, \\
& (\m{x}_k, \m{u}_k) \in \C{C}^1(\Omega) \times
\C{C}^0(\Omega).
\end{array} \right\}
\end{equation}
Since the function $\m{x}_0$ does not exist (there is no 0-th mesh interval),
we simply define $\m{x}_0(1) = \m{a}$, the initial condition.
The condition
\begin{equation}\label{continuity}
\m{x}_k(-1)=\m{x}_{k-1}(1)
\end{equation}
in (\ref{Dtau}) corresponds to the initial condition $\m{x}(0) = \m{a}$
when $k = 1$ and to continuity of the state across a mesh interval
boundary when $k > 1$.
Throughout the paper, (\ref{continuity}) is referred to as the
{\it continuity condition}.

In the $hp$-scheme developed in this paper,
the dynamics for $\m{x}_k$ are approximated by the Radau collocation
scheme developed in
\cite{DarbyHagerRao11,GargHagerRao11a,GargHagerRao10a,HagerHouRao15c}.
Let $\C{P}_N$ denote the space of polynomials of degree at most $N$
defined on the interval $\Omega$, and let $\C{P}_N^n$ denote the
$n$-fold Cartesian product $\C{P}_N \times \ldots \times \C{P}_N$.
We analyze a discrete approximation to (\ref{Dtau}) of the form
\begin{equation}\label{D}
\left. \begin{array}{cll}
\mbox {minimize} &C(\m{x}_K(1))&\\
\mbox {subject to} &\dot{\m{x}}_k(\tau_i)= h
\m{f}(\m{x}_k(\tau_i), \m{u}_{ki}),&1 \le i \le N, \;
\m{u}_{ki} \in \C{U}, \\
&\m{x}_k(-1)=\m{x}_{k-1}(1),& 1 \le k \le K, \; \m{x}_k \in \C{P}_N^n.
\end{array} \right\}
\end{equation}
Note that there is no polynomial associated with the control;
$\m{u}_{ki}$ corresponds to the value of the control at
$t_{k-1/2} + h\tau_i$.
In (\ref{D}) the dimension of $\C{P}_N$ is $N+1$
and there are $K$ mesh intervals, so a component of the state variable
is chosen from a space of dimension $K(N+1)$.
Similarly, there are $KN + K$ equations in (\ref{D}) corresponding to the
collocated dynamics at $KN$ points and the $K$ continuity conditions,
the initial condition at $t = 0$ and the $K-1$ continuity conditions
for the state at the interior mesh points.

For simplicity in the analysis, the same degree polynomials are
used in each mesh interval, while in practical implementations
of the $hp$-scheme
\cite{DarbyHagerRao11, DarbyHagerRao10, LiuHagerRao15, PattersonHagerRao14},
polynomials of different degrees are often used on different intervals.
On intervals where the solution is smooth, high degree polynomials are
employed, while on intervals where the solution is nonsmooth, low degree
polynomials are used.

We focus on a collocation scheme based on the
$N$ Radau quadrature points satisfying
\[
-1 < \tau_1 < \tau_2 < \ldots < \tau_N = 1 .
\]
The Radau points are related to the zeros of a Jacobi polynomial.
Recall that the Jacobi polynomials $P^{(\alpha, \beta)}_N$ are a class
of polynomials orthogonal with respect to the weight
$(1-\tau)^\alpha (1+\tau)^\beta$ for $\tau \in [-1, 1]$; the subscript $N$
specifies the polynomial degree.
The interior Radau abscissa $\tau_i$,
$1 \le i \le N-1$, are the zeros of the Jacobi polynomial
$P^{(1,0)}_{N-1}$ associated with the weight $1 - \tau$.
These quadrature points are sometimes called the flipped Radau points,
while the standard Radau points are $-\tau_i$, $1 \le i \le N$.
The analysis is the same for either set of points, while the notation is a
little cleaner for the flipped points.
Besides the $N$ collocation points, our analysis also utilizes
the noncollocated point $\tau_0 = -1$.

It is pointed out in \cite{HagerHouRao16c} that for a global collocation
scheme where $K = 1$, the discrete dynamics may be infeasible for certain
choices of $N$.
In contrast, the analysis in this paper implies that locally,
for each choice of the discrete control, there exists a unique discrete
state which satisfies the discrete dynamics when $K$ is sufficiently
large, or equivalently, when $h$ is sufficiently small, regardless of the
choice for $N$.
In this respect, the $hp$-collocation approach is more robust than
a global scheme.

Other global collocation schemes that have been presented in the literature
are based on the Lobatto quadrature points \cite{Elnagar1,Fahroo2},
on the Chebyshev quadrature points \cite{Elnagar4,FahrooRoss02},
on the Gauss quadrature points \cite{Benson2, GargHagerRao10a},
and on the extrema of Jacobi polynomials \cite{Williams1}.
Kang \cite{kang08,kang10} considers control systems in feedback
linearizable normal form, and shows that when the Lobatto discretized
control problem is augmented with bounds on the states and control, and
on certain Legendre polynomial expansion coefficients, then the
objectives in the discrete problem converge to the optimal objective
of the continuous problem at an exponential rate.
Kang's analysis does not involve coercivity assumptions for the continuous
problem, but instead imposes bounds in the discrete problem.
Also, in \cite{GongRossKangFahroo08} a consistency result is established
for a scheme based on Lobatto collocation.

Any of the global schemes
could be developed into an $hp$-collocation scheme.
Our rationale for basing our $hp$-scheme on the Radau collocation points
was the following:
In numerical experiments such as those in \cite{GargHagerRao10a},
there is often not much difference between the convergence speed of
approximations based on either Gauss or Radau collocation, while the
Lobatto scheme often converged much slower; and in some cases,
the Lobatto costate approximation did not converge due to a null space
that arises in the first-order optimality conditions --
see \cite{GargHagerRao10a}.
On the other hand, the implementation of an $hp$-scheme based on
the Radau quadrature points was much simpler than the implementation
based on the Gauss quadrature points.
The Gauss points lie in the interior of each mesh interval, which
requires the introduction of the state value at the mesh points.
Since one of the Radau points is a mesh point, there is no need
to introduce an additional noncollocated point.
The implementation ease of Chebyshev quadrature should be similar to
that of Gauss and was not pursued.
The $hp$-collocation scheme analyzed in this paper corresponds to the
scheme implemented in the popular GPOPS-II software package
\cite{Patterson2015} for solving optimal control problems.
This paper, in essence, provides a theoretical justification for
the algorithm implemented in the software.

For $\m{x} \in \C{C}^0(\Omega_0)$,
we use the sup-norm $\| \cdot \|_\infty$ given by
\[
\|\m{x}\|_\infty = \sup \{ |\m{x} (t)| : t \in \Omega_0 \} ,
\]
where $| \cdot |$ is the Euclidean norm.
Given $\m{y} \in \mathbb{R}^n$, the ball with center $\m{y}$ and radius
$\rho$ is denoted
\[
\C{B}_\rho (\m{y}) = \{ \m{x} \in \mathbb{R}^n :
|\m{x} - \m{y}| \le \rho \} .
\]
The following regularity assumption is assumed to hold throughout the paper.

{\bf Smoothness.}
The problem (\ref{P}) has a local minimizer
$(\m{x}^*, \m{u}^*)$ in
$\C{C}^1 (\Omega_0) \times \C{C}^0 (\Omega_0)$.
For some $\rho > 0$ and open set
$\C{O} \subset \mathbb{R}^{m+n}$ such that
\[
\C{B}_\rho (\m{x}^*(t),\m{u}^*(t)) \subset \C{O} \mbox{ for all }
t \in \Omega_0,
\]
the first two derivative of $f$ and $C$ are
Lipschitz continuous on the closure of
$\C{O}$ and on $\C{B}_\rho (\m{x}^*(1))$ respectively.

Let $\g{\lambda}^*$ denote the solution of the
linear costate equation
\begin{equation}\label{costate}
\dot{\g{\lambda}}^*(t)=-\nabla_xH({\m x}^*(t), {\m u}^*(t), {\g \lambda}^*(t)),
\quad {\g \lambda}^*(1)=\nabla C({\m x}^*(1)),
\end{equation}
where $H$ is the Hamiltonian defined by
$H({\m x}, {\m u}, {\g \lambda}) ={\g\lambda}\tr {\m f}({\m x}, {\m u})$ and
$\nabla$ denotes gradient.
By the first-order optimality conditions (Pontryagin's minimum principle),
we have
\begin{equation} \label{controlmin}
-\nabla_u H({\m x}^*(t), {\m u}^*(t), {\g \lambda}^*(t)) \in
N_{\C{U}}(\m{u}^*(t)) \mbox{ for all } t \in \Omega_0.
\end{equation}
For any $\m{u} \in \C{U}$,
\[
N_\C{U}(\m{u}) = \{ \m{w} \in \mathbb{R}^m :
\m{w}\tr(\m{v} - \m{u}) \le 0 \mbox{ for all } \m{v} \in \C{U} \},
\]
while $N_\C{U}(\m{u}) = \emptyset$ if $\m{u} \not\in \C{U}$.

We will show in Proposition~\ref{equiv} that the first-order
optimality conditions (Karush-Kuhn-Tucker conditions)
for (\ref{D}) are equivalent to the existence of
$\g{\lambda}_k \in \C{P}_{N-1}^n$, $1 \le k \le K$, such that
\begin{eqnarray}
\dot{\g \lambda}_{k}(\tau_i) &=&
-h\nabla_x H\left( {\m x}_k (\tau_i),{\m u}_{ki}, 
{\g \lambda}_k (\tau_i) \right), \quad
1 \leq i < N, \label{dcostate} \\[.05in]
\dot{\g \lambda}_k(1) &=&
-h\nabla_x H\left({\m x}_k (1),{\m u}_{kN}, 
{\g \lambda}_k (1) \right) +
\left(\g{\lambda}_k(1) - {\g \lambda}_{k+1}
(-1)\right)/\omega_N,
\label{dterminal}\\[.05in]
&& \quad \quad \mbox{where }
\g{\lambda}_{K+1}(-1) :=
\nabla C\left( \m{x}_{K}(1)\right) \nonumber \\
N_\C{U}(\m{u}_{ki}) &\ni&
-\nabla_u H\left({\m x}_k(\tau_i),{\m u}_{ki}, 
{\g \lambda}_k (\tau_i) \right),
\quad 1\leq i\leq N. \label{dcontrolmin}
\end{eqnarray}
Since the $K+1$ mesh interval does not exist,
(\ref{dterminal}) includes a definition for $\g{\lambda}_{K+1}(-1)$.
As we will see in Proposition~\ref{equiv},
$\g{\lambda}_k(-1)$ for $k \le K$ is the multiplier
associated with the continuity condition (\ref{continuity}).
Throughout the paper,
$\omega_i$, $1 \le i \le N$,
is the Radau quadrature weight associated with $\tau_i$.
The weight $\omega_i$ is the integral over $[-1, 1]$ of the
$i$-th Lagrange polynomial associated with the
Radau points $\tau_i$, $1 \le i \le N$.
This Lagrange polynomial of degree $N-1$ equals $1$ at the $i$-th Radau point
and $0$ at the other Radau points.
By \cite[Eq.~(3.134b)]{ShenTangWang11},
\[
\omega_i = \frac{2(1+\tau_i)}{[(1-\tau_i^2) \dot{P}_{N-1}^{(1,0)} (\tau_i)]^2},
\quad 1 \le i \le N-1, \quad \omega_N = \frac{2}{N^2},
\]
where $\dot{P}_{N-1}^{(1,0)} (\tau_i)$ is the derivative of the
Jacobi polynomial ${P}_{N-1}^{(1,0)}$ evaluated at it $i$-th zero.
Hence, the Radau quadrature weights are all positive.
Szeg\H{o} in \cite[Thm.~8.9.1]{Szego1939} provides tight estimates
for both the $\tau_i$ and the derivatives of the Jacobi polynomial at $\tau_i$
which yield a bound of the form
\[
\omega_i \le cN^{-1} \sqrt{1-\tau_i^2}, \quad 1 \le i < N.
\]
By \cite[Thm.~3.26]{ShenTangWang11},
\[
\int_{-1}^1 p(\tau) d\tau = \sum_{i=1}^N \omega_i p(\tau_i)
\]
for every $p \in \C{P}_{2N-2}$.
Taking $p = 1$, we see that the quadrature weights sum to 2.

Notice that the system (\ref{dcostate})--(\ref{dcontrolmin})
for the costate approximation does not contain a continuity
condition as in the primal discretization (\ref{D}), so the
costate approximation could be discontinuous across the mesh points.
Since $\C{P}_{N-1}$ has dimension $N$ and $1 \le k \le K$,
the approximation to a component of the costate has dimension $KN$,
while (\ref{dcostate})--(\ref{dterminal}) provides $KN$ equations.
Hence, if a continuity condition for the costate were imposed at the
mesh points, the system of equations (\ref{dcostate})--(\ref{dcontrolmin})
along with the continuity condition would be overdetermined.

The following two assumptions are utilized in the convergence analysis.
\begin{itemize}
\item[(A1)]
For some $\alpha > 0$, 
the smallest eigenvalue of the Hessian matrices $\nabla^2 C(\m{x}^*(1))$
and $ \nabla^2_{(x,u)} H(\m{x}^* (t), \m{u}^* (t), \g{\lambda}^* (t) )$
is greater than or equal to $\alpha$, uniformly for $t \in \Omega_0$.
\item[(A2)]
$K$ is large enough, or equivalently $h$ is small enough,
that $2hd_1 < 1$ and $2hd_2 < 1$, where
\begin{equation}\label{d1d2}
\hspace*{-.1in}
d_1 = \sup_{t \in \Omega_0} \|\nabla_x \m{f} (\m{x}^*(t), \m{u}^*(t))\|_\infty
\mbox{ and }
d_2 =
\sup_{t \in \Omega_0} \|\nabla_x \m{f} (\m{x}^*(t), \m{u}^*(t))\tr\|_\infty .
\end{equation}
Here $\| \cdot \|_\infty$ is the matrix sup-norm (largest absolute row sum).
\end{itemize}
\smallskip

The coercivity assumption (A1) ensures that the solution of the
discrete problem is a local minimizer.
The condition (A2) enters into the analysis of stability for
the perturbed dynamics;
as we will see, it ensures that for any choice of the discrete control,
there exists a unique choice for the discrete state that satisfies the
linearized dynamics.
In \cite[p. 804]{HagerHouRao16c}, where we analyze a Gauss collocation
scheme on a single interval, there is no $h$ in the analogue of (A2).
Hence, the convergence theory in \cite{HagerHouRao16c} only applies
to problems for which $\nabla_x \m{f} (\m{x}^*(t), \m{u}^*(t))$
is sufficiently small.
Consequently, the convergence theory for the $hp$-scheme is more robust
since it applies to a broader class of problems.

In addition to the two assumptions, the analysis utilizes four properties
of the Radau collocation scheme.
Let $\m{D}$ be the $N$ by $N+1$ matrix defined by
\begin{equation}\label{Ddef}
D_{ij} = \dot{L}_j (\tau_i), \;
\mbox{where }
L_j (\tau) := \prod^{N}_{\substack{l=0\\l\neq j}}
\frac{\tau-\tau_l}{\tau_j-\tau_l}, \;
1 \le i \le N \mbox{ and } 0 \le j \le N.
\end{equation}
The matrix $\m{D}$ is a differentiation matrix in the sense that
for any $\m{p} \in \mathbb{R}^{N+1}$,
$(\m{Dp})_i = \dot{p} (\tau_i)$, $1 \le i \le N$,
where $p \in \C{P}_N$ is the polynomial that satisfies
$p(\tau_j) = p_j$ for $0 \le j \le N$.
The submatrix $\m{D}_{1:N}$, consisting of the trailing $N$ columns of $\m{D}$,
has the following properties:
\smallskip
\begin{itemize}
\item [(P1)]
$\m{D}_{1:N}$ is invertible and
$\| \m{D}_{1:N}^{-1}\|_\infty = 2$.
\item [(P2)]
If $\m{W}$ is the diagonal matrix containing the Radau
quadrature weights $\g{\omega}$ on the diagonal, then the rows of the
matrix $[\m{W}^{1/2} \m{D}_{1:N}]^{-1}$ have Euclidean norm bounded by
$\sqrt{2}$.
\end{itemize}
\smallskip
The proof of (P1) and (P2) are given in Appendix~1.


There is a related matrix that enters into the convergence analysis of
the $hp$-scheme.
Let $\m{D}^\ddagger$ be the $N$ by $N$ matrix defined by
\begin{equation}\label{Ddager}
D_{ij}^\ddagger = - \left( \frac{\omega_j}{\omega_i} \right)
D_{ji}, \quad 1 \le i \le N, \quad 1 \le j \le N .
\end{equation}
The matrix $\m{D}^\ddagger$ arises in the analysis of the costate equation.
In Section~4.2.1 of \cite{GargHagerRao10a}, we introduce a matrix
$\m{D}^\dagger$ which is a differentiation matrix for the collocation
points $\tau_i$, $1 \le i \le N$.
That is, if $p$ is a polynomial of degree at most $N-1$ and
$\m{p}$ is the vector with components $p(\tau_i)$, $1 \le i \le N$,
then $(\m{D}^\dagger \m{p})_i = \dot{p}(\tau_i)$.
The matrix $\m{D}^\ddagger$ only differs from $\m{D}^\dagger$ in a
single entry: $D^\ddagger_{NN} = D^\dagger_{NN} - 1/\omega_N$.
As a result,
\begin{equation}\label{h282}
(\m{D}^{\ddag} \m{p})_i = \dot{p}(\tau_i), \quad
1 \le i < N, \quad
(\m{D}^{\ddag} \m{p})_N = \dot{p}(\tau_N) - p(1)/\omega_N.
\end{equation}
If $\m{D}^{\ddag} \m{p} = \m{0}$, then
$\dot{p}(\tau_i) = 0$ for $i < N$ by the first equality in (\ref{h282}).
Since $\dot{p}$ has degree $N-2$ and it vanishes at $N-1$ points,
$\dot{p}$ is identically zero and $p$ is constant.
By the final equation in (\ref{h282}), $p(1) = 0$ when
$\m{D}^\ddag \m{p} = \m{0}$, which implies that $p$ is identically zero.
This shows that $\m{D}^\ddagger$ is invertible.
We find that $\m{D}^\ddagger$ has the following properties:
\smallskip
\begin{itemize}
\item [(P3)]
$\m{D}^\ddagger$ is invertible and
$\| (\m{D}^\ddagger)^{-1}\|_\infty \le 2$.
\item [(P4)]
The rows of the
matrix $[\m{W}^{1/2} \m{D}^\ddagger]^{-1}$ have Euclidean norm bounded by
$\sqrt{2}$.
\end{itemize}
\smallskip

In Proposition~\ref{deriv_exact} at the end of the paper,
an explicit formula is given for the inverse of $\m{D}^\ddagger$.
However, it is not clear from the formula that
$\| (\m{D}^\ddagger)^{-1}\|_\infty$ is bounded by 2.
It is shown in Appendix~1, in inequality (\ref{Euclidean}),
that (P2) implies (P1).
By the same inequality, (P4) implies (P3).
Unlike (P1) where the norm $\| \m{D}_{1:N}^{-1}\|_\infty$ is 2 as shown in
Lemma~\ref{P1P2}, it is observed numerically that the norm
$\| (\m{D}^\ddagger)^{-1}\|_\infty$ is strictly less than 2,
and it approaches 2 in the limit as $N$ tends to infinity.
In the Appendix, we observe that for $N$ up to 300,
these norms increase monotonically towards the given bounds.
Again, the proof of (P4) for general $N$ is currently open.
Properties (P1)--(P4) differ from the assumptions (A1)--(A2)
in the sense that (A1)--(A2) only hold for certain control problems; while
(P1)--(P4) seem to hold in general.

In the analysis of the Gauss scheme \cite{HagerHouRao16c},
properties (P3) and (P4) follow immediately from (P1) and (P2)
since the analogue of $\m{D}^\ddagger$ in
\cite{HagerHouRao16c} is related to $\m{D}_{1:N}$ through an
exchange operation.
However, due to the asymmetry of the Radau collocation points and the
lower degree of the polynomials in the discrete adjoint system
(\ref{dcostate})--(\ref{dcontrolmin}), a corresponding relationship between
$\m{D}^\ddagger$ and $\m{D}_{1:N}$ in the Radau scheme does not seem to hold.
Nonetheless, the bounds in (P3) and (P4) are observed to be the same as the
bounds in (P1) and (P2).

Given a local minimizer $(\m{x}^*, \m{u}^*)$ of (\ref{P}), let
$\m{x}_k^*$, $\m{u}_k^*$, and $\g{\lambda}_k^*$ be the
state, control, and costate associated with the mesh interval $[t_{k-1}, t_k]$
and the change of variables $t =$ $t_{k-1/2} + h \tau$, and define
$t_{kj} = t_{k-1/2} + h \tau_j$.
The domain of
$\m{x}_k^*$, $\m{u}_k^*$, or $\g{\lambda}_k^*$ is $[-1, +1]$ where
$-1$ corresponds to $t_{k-1}$ and $+1$ corresponds to $t_k$.
We define the following related discrete variables:
\begin{equation}\label{optdiscrete}
\left. \begin{array}{lll}
\m{X}_{kj}^* = \m{x}_k^*(\tau_j) = \m{x}^*(t_{kj}),
& 0 \le j \le N, & 1 \le k \le K, \\
\m{U}_{kj}^* = \m{u}_k^*(\tau_j) = \m{y}^*(t_{kj}),
& 1 \le j \le N, & 1 \le k \le K, \\
\g{\Lambda}_{kj}^* = \g{\lambda}_k^*(\tau_j) = \g{\lambda}^*(t_{kj}),
& 0 \le j \le N, & 1 \le k \le K.
\end{array}
\right\}
\end{equation}

Suppose that $\m{x}_k^N \in \C{P}_N^n$, $1 \le k \le K$,
is a polynomial which is a stationary point of (\ref{D})
for some discrete controls $\m{u}_k^N$, and suppose that
$\g{\lambda}_k^N \in \C{P}_{N-1}^n$ satisfy
(\ref{dcostate})--(\ref{dcontrolmin}).
We define the following related discrete variables:
\[
\begin{array}{lll}
\m{X}_{kj}^N = \m{x}_k^N(\tau_j), & 0 \le j \le N, & 1 \le k \le K, \\
\m{U}_{kj}^N = \m{u}_{kj}^N, & 1 \le j \le N, & 1 \le k \le K, \\
\g{\Lambda}_{kj}^N = \g{\lambda}_k^N(\tau_j), & 0 \le j \le N, & 1 \le k \le K.
\end{array}
\]
Thus capital letters always refer to discrete variables.
As noted earlier, the costate polynomials associated with the
discrete problem are typically discontinuous across the mesh points,
and $\g{\Lambda}_{kN}^N \ne \g{\Lambda}_{k+1,0}^N$.

The convergence analysis only involves the smoothness
of the optimal state and associated costate on the interior
of each mesh interval.
Let $\C{H}^p(a, b)$ denote the Sobolev space of functions
with square integrable derivatives on $(a, b)$ through order $p$.
Let $\C{PH}^p(\Omega_0)$ denote the space of continuous
functions whose restrictions to $(t_{k-1}, t_k)$ are contained in
$\C{H}^p(t_{k-1}, t_k)$ for each $k$ between 1 and $K$
(piecewise $\C{H}^p$).
The norm on $\C{PH}^p(\Omega_0)$ is the same as the
norm on $\C{H}^p(\Omega_0)$ except that the integral is
computed over the interior of each mesh interval.
In this paper, the error bounds are expressed in terms of a seminorm
$| \cdot |_{\C{PH}^p(\Omega_0)}$ which only involves
the $p$-th order derivative:
\[
|\m{x}|_{\C{PH}^p(\Omega_0)} =
\left( \sum_{k=1}^K \int_{t_{k-1}}^{t_k} \left|\frac{d^p \m{x} (t)}
{dt^p} \right|^2 \; dt \right)^{1/2}.
\]
%
%
The following convergence result relative to the vector sup-norm
(largest absolute element) will be established.
\smallskip
\begin{theorem}\label{maintheorem}
If $(\m{x}^*, \m{u}^*)$ is a local minimizer for the continuous problem
$(\ref{P})$ with $\m{x}^*$ and
$\g{\lambda}^* \in \C{PH}^\eta (\Omega_0)$ for some $\eta \ge 2$,
and {\rm (A1)}, {\rm (A2)}, and {\rm (P4)} hold,
then for $N$ sufficiently large or for $h$ sufficiently small with $N \ge 2$,
the discrete problem $(\ref{D})$ has a local minimizer
and associated multiplier satisfying
$(\ref{dcostate})$--$(\ref{dcontrolmin})$, and we have
\begin{eqnarray}
&\max \left\{ \left\|{\bf X}^{N}-{\bf X}^*\right\|_\infty ,
\left\|{\bf U}^{N}-{\bf U}^{*}\right\|_\infty,
\left\|{\g \Lambda}^{N}-{\g \Lambda}^{*}\right\|_\infty \right\}
\nonumber\\
&\le h^{p-1}\left(\frac{c}{N}\right)^{p-1}
|{\mathbf{x}^{*}}|_{\C{PH}^p(\Omega_0)}
+ h^{q-1}\left(\frac{c}{N}\right)^{q-1.5}
|{\lmb^{*}}|_{\C{PH}^q(\Omega_0)},
\label{maineq}
\end{eqnarray}
where $p=\min(\eta, N+1)$, $q=\min(\eta, N)$,
and $c$ is independent of $h$, $N$, and $\eta$.
\end{theorem}
\smallskip

The proof of Theorem~\ref{maintheorem} begins in Section~\ref{abstract}
where the discrete first-order optimality conditions are formulated
as an inclusion of the form
$\C{T}(\m{X}, \m{U}, \g{\Lambda}) \in \C{F}(\m{U})$.
In Section~\ref{residual} a bound is obtained for the distance $d^*$ from
$\C{T}(\m{X}^*, \m{U}^*, \g{\Lambda}^*)$ to $\C{F}(\m{U}^*)$, where
$(\m{X}^*, \m{U}^*, \g{\Lambda}^*)$ denotes the optimal discrete variables
defined in (\ref{optdiscrete}).
This bound is based on an estimate given in Section~\ref{sect_interp}
for the $\C{H}^1$ approximation error of the polynomial that
interpolates $\m{x}^*$ at $\tau_i$, $0 \le i \le N$.
The remainder of the paper focuses on showing that the bound for $d^*$
is also a bound for the distance from $(\m{X}^*, \m{U}^*, \g{\Lambda}^*)$
to a solution of the inclusion
$\C{T}(\m{X}, \m{U}, \g{\Lambda}) \in \C{F}(\m{U})$.
The analysis is based on Proposition~\ref{prop}, where it is shown that such
a bound can be obtained if a linearized version of the original inclusion
is stable under perturbations.
More precisely, we need to show that the problem of finding
$(\m{X}, \m{U}, \g{\Lambda})$ such that
\[
\nabla \C{T}(\m{X}^*, \m{U}^*, \g{\Lambda}^*)[\m{X}, \m{U}, \g{\Lambda}] +
\m{Y} \in \C{F}(\m{U})
\]
has a unique solution which depends Lipschitz continuously
on the perturbation $\m{Y}$.
This analysis, which utilizes
assumptions (A1)--(A2) and properties (P1)--(P4),
begins in Section~\ref{inverse} where perturbations
in the linearized state and costate discrete dynamics are analyzed.
In Section~\ref{Invert} it is shown that solving the linearized inclusion is
equivalent to solving a quadratic program,
where perturbations in the inclusion appear as
linear terms in the quadratic program;
the strong convexity assumption (A1) implies the
existence of a unique solution to the quadratic program, which in turn
implies the existence of a unique solution to the inclusion.
Finally, in Section~\ref{Lip} the unique solution of the linearized inclusion
is shown to depend Lipschitz continuously on the perturbation.
This Lipschitz property and the bound for $d^*$ are combined with
Proposition~\ref{prop} to obtain (\ref{maineq}).
The tightness and possible extensions
of the error bound (\ref{maineq}) are explored in
Section~\ref{numerical} using some problems with known solutions.
In the proof of Theorem~\ref{maintheorem}, we need to make the right
side of (\ref{maineq}) sufficiently small to establish the existence
of the claimed solution to the discrete problem.
The conditions $\eta \ge 2$ and $N \ge 2$ in the statement of the
theorem ensure that as $h$ goes to zero, $h^{p-1}$ and $h^{q-1}$ go to zero,
and as $N$ tends to infinity, $(c/N)^{p-1}$ and $(c/N)^{q-1.5}$ go to zero.

Since the discrete costate could be discontinuous across a mesh point,
Theorem~\ref{maintheorem} implies convergence of the discrete costate
on either side of the mesh point to the continuous costate at the mesh point.
The discrete problem provides an estimate for the optimal control at $t = 1$
in the continuous problem,
but not at $t = 0$ since this is not a collocation point.
Due to the strong convexity assumption (A1),
an estimate for the discrete control at $t = 0$ can be obtained from
the minimum principle (\ref{controlmin}) since the initial state
is given, while we have an estimate for the associated costate at $t = 0$.
Alternatively, polynomial interpolation could be used to obtain
estimates for the optimal control at $t = 0$.

In a recent paper \cite{HagerMohapatraRao16b}, where we analyze a Gauss
collocation scheme on a single interval, $p = q = \min(\eta, N+1)$.
The differences between Radau and Gauss collocation are
due to the asymmetry of the Radau points, and the asymmetry in the
Radau first-order optimality conditions; that is, for the Radau points,
$\g{\lambda}_k \in \C{P}_{N-1}^n$
while $\m{x}_k \in \C{P}_{N}^n$. 

{\bf Notation.}
We let $\Omega$ denote the interval $[-1, 1]$, while
$\Omega_0$ is the interval $[0, 1]$.
Let $\C{P}_N$ denote the space of polynomials of degree at most $N$,
while $\C{P}_N^0$ is the subspace consisting of polynomials in
$\C{P}_N$ that vanish at $t = -1$ and $t = 1$.
The meaning of the norm $\| \cdot \|_\infty$ is based on context.
If $\m{x} \in \C{C}^0 (\Omega)$, then
$\|\m{x}\|_\infty$ denotes the maximum of $|\m{x}(t)|$ over
$t \in [-1, 1]$, where $| \cdot|$ is the Euclidean norm.
For a vector $\m{v} \in \mathbb{R}^m$,
$\|\m{v}\|_\infty$ is the maximum of $|v_i|$ over $1 \le i \le m$.
If $\m{A} \in \mathbb{R}^{m \times n}$, then $\|\m{A}\|_\infty$
is the largest absolute row sum (the matrix norm induced by the
$\ell_\infty$ vector norm).
We let $|\m{A}|$ denote the matrix norm induced by the Euclidean vector norm.
Throughout the paper, the index $k$ is used for the mesh interval,
while the indices $i$ and $j$ are associated with collocation points.
If $\m{p} \in \mathbb{R}^{KNn}$, then $\m{p}_k$ for
$1 \le k \le K$ refers to vector with components $\m{p}_{kj} \in \mathbb{R}^n$,
for $1 \le j \le N$.
The dimension of the identity matrix $\m{I}$ is often clear from context;
when necessary, the dimension of $\m{I}$ is specified by a subscript.
For example, $\m{I}_n$ is the $n$ by $n$ identity matrix.
The gradient is denoted $\nabla$, while $\nabla^2$ denotes the Hessian;
subscripts indicate the differentiation variables.
Throughout the paper, $c$ is a generic constant which has different
values in different equations.
The value of $c$ is always independent of $h$, $N$, and $\eta$.
The vector $\m{1}$ has all entries equal to one, while
the vector $\m{0}$ has all entries equal to zero;
again, their dimension should be clear from context.
If $\m{D}$ is the differentiation matrix introduced in (\ref{Ddef}), then
$\m{D}_j$ is the $j$-th column of $\m{D}$ and
$\m{D}_{i:j}$ is the submatrix formed by columns $i$ through $j$.
We let $\otimes$ denote the Kronecker product.
If $\m{U} \in \mathbb{R}^{m \times n}$ and $\m{V} \in \mathbb{R}^{p \times q}$,
then $\m{U} \otimes \m{V}$ is the $mp$ by $nq$ block matrix whose
$(i,j)$ block is $u_{ij} \m{V}$.
We let $\C{L}^2 (\Omega)$ denote the usual space of square integrable functions
on $\Omega$, while $\C{H}^p(\Omega)$ is the Sobolev space consisting
of functions with square integrable derivatives through order $p$.
The seminorm in $\C{H}^p(\Omega)$, corresponding to the $\C{L}^2(\Omega)$
norm of the $p$-order derivatives, is denoted $| \cdot |_{\C{H}^p(\Omega)}$.
The subspace of $\C{H}^1(\Omega)$ corresponding to functions that vanish
at $t = -1$ and $t = 1$ is denoted $\C{H}_0^1(\Omega)$.

\section{Abstract setting}
\label{abstract}
Given a feasible point for the discrete problem (\ref{D}),
define $\m{X}_{kj} = \m{x}_k(\tau_j)$ and $\m{U}_{ki} = \m{u}_{ki}$.
As noted earlier, $\m{D}$ is a differentiation matrix in the sense that
\[
\sum_{j=0}^N D_{ij} \m{X}_{kj} = \dot{\m{x}}_k (\tau_i), \quad
1 \le i \le N.
\]
Hence, the discrete problem (\ref{D}) can be reformulated as
\begin{equation}\label{nlp}
\left. \begin{array}{cll}
\mbox {minimize} &C(\m{X}_{KN})&\\
\mbox {subject to} & \sum_{j=0}^N D_{ij} \m{X}_{kj} =
h \m{f}(\m{X}_{ki}, \m{U}_{ki}), \quad \m{U}_{ki} \in \C{U}, &1 \le i \le N,\\
&\m{X}_{k0} =\m{X}_{k-1,N},& 1 \le k \le K,
\end{array} \right\}
\end{equation}
where $\m{X}_{0N} = \m{a}$, the starting condition.

We introduce multipliers $\g{\mu}_{ki}$ associated with the
constraints in (\ref{nlp}) and write the Lagrangian as
\begin{eqnarray*}
&\mathcal{L}({\g \mu}, {\bf X},{\bf U}) =& \\
&C\left(\mathbf{X}_{KN}\right)
+\sum_{k=1}^K\sum_{i=1}^N\left\langle\g{\mu}_{ki},
h{\bf f}({\bf X}_{ki},{\bf U}_{ki})
-\sum_{j=0}^{N}{D}_{ij}{\bf X}_{kj} \right\rangle
 + \sum_{k=1}^K\left\langle{\g \mu}_{k0}, 
\left({\bf X}_{k-1,N}-{\bf X}_{k0}\right)\right\rangle.&
\end{eqnarray*}
The first-order optimality conditions for \eqref{nlp},
often called the Karush-Kuhn-Tucker (KKT) conditions, lead to the
following relations (we show the variable with which we differentiate
the Lagrangian followed by the associated condition):
\begin{eqnarray}
\m{X}_{k0}\;&\Rightarrow&
\sum_{i=1}^N{D}_{i0}{\g \mu}_{ki}=-\g{\mu}_{k0},
\label{NC0} \\
\m{X}_{kj}\;&\Rightarrow&
\sum_{i=1}^N {D}_{ij}\g{\mu}_{ki} = h\nabla_x H
({\bf X}_{kj},{\bf U}_{kj},\g{\mu}_{kj}), \quad 1 \le j < N, \label{NC1} \\
\m{X}_{kN}&\Rightarrow&
\sum_{i=1}^N{D}_{iN}{\g \mu}_{ki} =
h \nabla_x H({\bf X}_{kN},{\bf U}_{kN},\g{\mu}_{kN})
+ \g{\mu}_{k+1,0}, \label{NC2} \\
&& \quad \quad \g{\mu}_{K+1,0} := \nabla C(\m{X}_{KN}), \label{NC2b} \\
\m{U}_{ki}\;&\Rightarrow&
-\nabla_u H\left({\bf X}_{ki},{\bf U}_{ki},{\g \mu}_{ki}\right)
\in N_{\C{U}}(\m{U}_{ki}).
\label{NC3}
\end{eqnarray}
%


We first relate the KKT multipliers in (\ref{NC0})--(\ref{NC3})
to the polynomials satisfying (\ref{dcostate})--(\ref{dcontrolmin}).
\begin{proposition}\label{equiv}
The multipliers $\g{\mu}_k \in \mathbb{R}^{Nn}$
satisfy $(\ref{NC0})$--$(\ref{NC3})$ if and only if the polynomial
$\g{\lambda}_k \in \C{P}_{N-1}^n$ given by
$\g{\lambda}_k(\tau_i) = \g{\mu}_{ki}/\omega_i$,
$1 \le i \le N$, satisfies $(\ref{dcostate})$--$(\ref{dcontrolmin})$.
Moreover, $\g{\mu}_{k0} = \g{\lambda}_k(-1)$.
\end{proposition}

\begin{proof}
We start with multipliers $\g{\mu}_k$ satisfying (\ref{NC0})--(\ref{NC3}).
Define $\g{\Lambda}_{ki} = \g{\mu}_{ki}/\omega_i$ for $1 \le i \le N$,
and let $\g{\lambda}_k \in \C{P}_{N-1}^n$ be the polynomial that satisfies
$\g{\lambda}_k(\tau_i) = \g{\Lambda}_{ki}$.
Also, set $\g{\Lambda}_{k0} = \g{\mu}_{k0}$.
In terms of $\g{\Lambda}_{ki}$ and the matrix
$D_{ij}^\ddagger = - \omega_j D_{ji}/\omega_i$,
the equations (\ref{NC1}), (\ref{NC2}), and (\ref{NC3}) become
\begin{eqnarray}
\sum_{j=1}^N {D}_{ij}^\ddagger\g{\Lambda}_{kj} &=& -h\nabla_x H
({\bf X}_{ki},{\bf U}_{ki},\g{\Lambda}_{ki}), \quad 1 \le i < N, \label{nc1} \\
\sum_{i=1}^N{D}_{Ni}^\ddagger{\g \Lambda}_{ki} &=&
-[h \nabla_x H({\bf X}_{kN},{\bf U}_{kN},\g{\Lambda}_{kN})
+ \g{\Lambda}_{k+1,0}/\omega_N],
\label{nc2} \\
N_{\C{U}}(\m{U}_{ki}) &\ni&
-\nabla_u H({\bf X}_{ki},{\bf U}_{ki},\g{\Lambda}_{ki}), \quad
1 \le i \le N.
\label{nc3}
\end{eqnarray}
%

Since the polynomial that is identically equal to $\m{1}$
has derivative $\m{0}$ and since $\m{D}$ is a differentiation matrix,
we have $\m{D1} = \m{0}$,
which implies that ${\bf D}_{0}=-\sum_{j=1}^N{\bf D}_j$,
where $\m{D}_j$ is the $j$-th column of $\m{D}$.
Hence, the first definition in \eqref{NC0} can be written
\begin{eqnarray}
\g{\Lambda}_{k0} &=&
-\sum_{i=1}^N\g{\mu}_{ki} D_{i0}
=\sum_{i=1}^N\sum_{j=1}^N\g{\mu}_{ki}D_{ij} =
\sum_{i=1}^N\sum_{j=1}^N \omega_j
\left( \frac{\g{\mu}_{ki}}{\omega_i} \right)
(\omega_i D_{ij}/\omega_j) \nonumber \\
&=& -\sum_{i=1}^N\sum_{j=1}^N \omega_i D_{ij}^\ddag\g{\Lambda}_{kj}
\label{eq06} \\
&=& \g{\Lambda}_{k+1,0} + h
\sum_{i=1}^N \omega_i \nabla_x H(\m{X}_{ki}, \m{U}_{ki}, \g{\Lambda}_{ki}) ,
\label{nc0}
\end{eqnarray}
where (\ref{nc0}) is due to (\ref{nc1})--(\ref{nc2}).

%
As noted in (\ref{h282}),
\begin{eqnarray}
\sum_{j=1}^N D_{ij}^\ddag\g{\Lambda}_{kj} &=&
\dot{\g{\lambda}}_k(\tau_i), \quad 1 \le i < N, \quad \mbox{and}\label{h-2} \\
\sum_{j=1}^N D_{Nj}^\ddag\g{\Lambda}_{kj} &=&
\dot{\g{\lambda}}_k (1) - \g{\lambda}_k(1)/\omega_N. \label{h-1}
\end{eqnarray}
This substitution in (\ref{eq06}) yields
\begin{equation}\label{h0}
\g{\Lambda}_{k0} =
\g{\lambda}_k (1) - \sum_{i=1}^N \omega_i \dot{\g{\lambda}}_k(\tau_i) .
\end{equation}
Since $\dot{\g{\lambda}}_k \in \C{P}_{N-2}^n$ and
$N$-point Radau quadrature is exact for these polynomial, we have
\begin{equation}\label{h1}
\sum_{i=1}^N \omega_i \dot{\g{\lambda}}_k(\tau_i) =
\int_{-1}^{1} \dot{\g{\lambda}}_k (\tau) d\tau =
\g{\lambda}_k(1) - \g{\lambda}_k(-1) .
\end{equation}
Combine (\ref{h0}) and (\ref{h1}) to obtain
\begin{equation}\label{muk}
\g{\Lambda}_{k0} = \g{\lambda}_k (-1).
\end{equation}

Let $\m{x}_k \in \C{P}_N^n$ be the polynomial that satisfies
$\m{x}_k (\tau_j) = \m{X}_{kj}$ for all $0 \le j \le N$.
By (\ref{h-2}), (\ref{dcostate}) is equivalent to (\ref{nc1})
which is equivalent to (\ref{NC1}) after a change of variables.
By (\ref{h-1}) and (\ref{muk}), (\ref{dterminal})
is equivalent to (\ref{nc2}), which is equivalent to (\ref{NC2}) after
a change of variables.
Finally, (\ref{dcontrolmin}) is the same as
(\ref{nc3}) which is equivalent to (\ref{NC3}) after a change of variables.
The equivalence between $\g{\Lambda}_{k0}$ and $\g{\lambda}_k(-1)$ was
derived in (\ref{muk}).
This shows that the polynomial $\g{\lambda}_k(\tau)$ satisfies
(\ref{dcostate})--(\ref{dcontrolmin}).
The converse of the proposition follows by reversing all the steps
in the derivation.
\end{proof}

The dynamics for (\ref{nlp}),
the first-order optimality conditions (\ref{nc1})--(\ref{nc3}),
the formula (\ref{nc0}) for $\g{\Lambda}_{k0}$,
and the terminal costate condition (\ref{NC2b})
can be written as $\C{T}(\m{X}, \m{U}, \g{\Lambda}) \in \C{F}(\m{U})$ where
\begin{eqnarray}
\C{T}_{1ki}(\m{X}, \m{U}, \g{\Lambda}) &=&
\left( \sum_{j=0}^{N}{D}_{ij}{\bf X}_{kj} \right)
-h {\bf f}({\bf X}_{ki},{\bf U}_{ki}),
\quad 1\leq i\leq N ,\label{T1}\\
\C{T}_{2k}(\m{X}, \m{U}, \g{\Lambda}) &=&
{\bf X}_{k0}-\m{X}_{k-1,N}, \label{T2}\\
\C{T}_{3ki}(\m{X}, \m{U}, \g{\Lambda}) &=&
\left( \sum_{j=1}^{N}{D}_{ij}^\ddagger{\g \Lambda}_{kj} \right) +
h \nabla_x H({\bf X}_{ki},{\bf U}_{ki}, {\g \Lambda}_{ki}),
\quad 1 \leq i < N , \label{T3}\\
\C{T}_{3kN}(\m{X}, \m{U}, \g{\Lambda}) &=&
\sum_{j=1}^{N}{D}_{Nj}^\ddagger{\g \Lambda}_{kj}
+h \nabla_x H\left({\m X}_{kN},{\m U}_{kN}, {\g \Lambda}_{kN}\right)
+\g{\Lambda}_{k+1,0}/\omega_N,\label{TN}\\
\C{T}_{4k}(\m{X}, \m{U}, \g{\Lambda}) &=&
\g{\Lambda}_{k0} - \g{\Lambda}_{k+1,0} - h \sum_{i=1}^N
\omega_i \nabla_x H(\m{X}_{ki}, \m{U}_{ki}, \g{\Lambda}_{ki}),\label{T4}\\
\C{T}_{5}(\m{X}, \m{U}, \g{\Lambda}) &=&
\nabla C(\m{X}_{KN}) - \g{\Lambda}_{K+1,0}, \label{T5}\\
\C{T}_{6ki}(\m{X}, \m{U}, \g{\Lambda}) &=&
-h \nabla_u H({\bf X}_{ki}, {\bf U}_{ki}, {\g \Lambda}_{ki}),
\quad 1\leq i\leq N , \label{T6}
\end{eqnarray}
%
%
where $1 \le k \le K$.
The initial state is 
$\m{X}_{0N} = \m{X}_{10} = \m{a}$.
The components of $\C{F}$ are given by
\[
\C{F}_1 = \C{F}_2 = \C{F}_3 = \C{F}_4 = \C{F}_5 = \m{0}, \quad
\mbox{and }\C{F}_{6ki}(\m{U}) = N_\C{U}(\m{U}_{ki}).
\]

The proof of Theorem~\ref{maintheorem} is based on
\cite[Proposition 3.1]{DontchevHagerVeliov00}, given below in
a slightly simplified form.
Other results like this are contained in
Theorem~3.1 of \cite{DontchevHager97},
in Proposition~5.1 of \cite{Hager99c}, in Theorem~2.1 of \cite{Hager02b}, and
in Theorem~1 of \cite{Hager90}.
\smallskip
\begin{proposition}\label{prop}
Let ${\cal X}$ be a Banach space and
let ${\cal Y}$ be a linear normed space with the norms in both
spaces denoted $\| \cdot \|$.
Let ${\cal F} : {\cal X} \mapsto 2^{\cal Y}$
and let ${\cal T}: {\cal X} \mapsto {\cal Y}$ with $\cal T$
continuously Fr\'{e}chet differentiable
in $B_r(\g{\theta}^*)$, the ball with center $\g{\theta}^*$ and radius $r$,
for some $\g{\theta}^* \in {\cal X}$ and $r > 0$.
Suppose that the following conditions hold for some
$\g{\delta} \in {\cal Y}$ and scalars $\epsilon$ and $\gamma > 0$:
\begin{itemize}
\item[{\rm (C1)}]
${\cal T}(\g{\theta}^*) + \g{\delta} \in {\cal F}(\g{\theta}^*)$.
\item[{\rm (C2)}]
$\|\nabla {\cal T}(\g{\theta}) - \nabla \C{T}(\g{\theta}^*)\| \le \epsilon$
for all $\g{\theta} \in B_r(\g{\theta}^*)$.
\item[{\rm (C3)}]
The map $({\cal F} - \nabla \C{T}(\g{\theta}^*))^{-1}$ is single-valued
and Lipschitz continuous with Lipschitz constant $\gamma$.
\end{itemize}
If $\epsilon \gamma < 1$ and
$\|\g{\delta}\| \le (1 - \gamma \epsilon )r/\gamma$,
then there exists a unique $\g{\theta} \in B_r(\g{\theta}^*)$ such that
${\cal T}(\g{\theta}) \in {\cal F}(\g{\theta})$.
Moreover, we have the estimate
\begin{equation} \label{abs}
\|\g{\theta} - \g{\theta}^*\| \leq
\frac{\gamma}{1 - \gamma \epsilon} \|\g{\delta} \| .
\end{equation}
\end{proposition}
\smallskip
\begin{proof}
Define
$\Phi(\g{\theta}) =$
$[\C{F} - \nabla\C{T}(\g{\theta}^*)]^{-1}
[\C{T}-\nabla\C{T}(\g{\theta}^*)](\g{\theta})$.
For all $\g{\theta}_1$ and $\g{\theta}_2 \in B_r(\g{\theta}^*)$,
a Taylor expansion with integral remainder term yields
\[
[\C{T} - \nabla\C{T}(\g{\theta}^*)](\g{\theta}_2) =
[\C{T} - \nabla\C{T}(\g{\theta}^*)](\g{\theta}_1) +
\int_0^1 [\nabla \C{T} (\g{\theta}_1 +
s(\g{\theta}_2-\g{\theta}_1)) - \nabla\C{T}(\g{\theta}^*)]
\; ds \; (\g{\theta}_2 - \g{\theta}_1).
\]
By (C2), it follows that
\begin{equation}\label{L1}
\|[\C{T} - \nabla\C{T}(\g{\theta}^*)])(\g{\theta}_2) -
[\C{T} - \nabla\C{T}(\g{\theta}^*)](\g{\theta}_1)\| \le
\epsilon \|\g{\theta}_2 - \g{\theta}_1\|.
\end{equation}
By (C3) and (\ref{L1}), we have
\begin{eqnarray*}
&& \|\Phi(\g{\theta}_1) - \Phi(\g{\theta}_2)\| \\
&=& \|[\C{F} - \nabla\C{T}(\g{\theta}^*)]^{-1}
[\C{T} - \nabla\C{T}(\g{\theta}^*)](\g{\theta}_1)
- [\C{F} - \nabla\C{T}(\g{\theta}^*)]^{-1}
[\C{T} - \nabla\C{T}(\g{\theta}^*)](\g{\theta}_2)\| \\
&\le& \gamma \| [\C{T} - \nabla\C{T}(\g{\theta}^*)]
(\g{\theta}_1) - [\C{T} - \nabla\C{T}(\g{\theta}^*)](\g{\theta}_2)\| \\
&\le& \epsilon \gamma \|\g{\theta}_1 - \g{\theta}_2\| .
\end{eqnarray*}
Since $\epsilon \gamma < 1$, $\Phi$ is a contraction on $B_r (\g{\theta}^*)$.
Subtracting $\nabla\C{T}(\g{\theta}^*)(\g{\theta}^*)$ from each side of (C1)
gives
\[
[\C{T}-\nabla\C{T}(\g{\theta}^*)]\g{\theta}^* + \g{\delta}
\in [\C{F} - \nabla\C{T}(\g{\theta}^*)](\g{\theta}^*),
\]
and utilizing the uniqueness in (C3) yields
\[
\g{\theta}^* =
[\C{F} - \nabla\C{T}(\g{\theta}^*)]^{-1}
[(\C{T}-\nabla\C{T}(\g{\theta}^*))\g{\theta}^* + \g{\delta}].
\]
With this substitution, it follows from (\ref{L1}), (C3), and (C2) that
\begin{eqnarray}
&& \|\Phi (\g{\theta}) - \g{\theta}^*\|\nonumber\\
&=&
\|[\C{F} - \nabla\C{T}(\g{\theta}^*)]^{-1}
[\C{T} - \nabla\C{T}(\g{\theta}^*)](\g{\theta})
- [\C{F} - \nabla\C{T}(\g{\theta}^*)]^{-1}[(\C{T} -
\nabla\C{T}(\g{\theta}^*))(\g{\theta}^*) + \g{\delta}]\| \nonumber\\
&\le& \gamma
\|[\C{T} - \nabla\C{T}(\g{\theta}^*)](\g{\theta})
- [\C{T} - \nabla\C{T}(\g{\theta}^*)](\g{\theta}^*) - \g{\delta}]\| \nonumber\\
&\le& \gamma (\epsilon \|\g{\theta} - \g{\theta}^*\| + \|\g{\delta}\|)
\le \gamma (\epsilon r + \|\g{\delta}\|) \label{abs1}
\end{eqnarray}
for all $\g{\theta} \in B_r(\g{\theta}^*)$.
The assumption that
$\|\g{\delta}\| \le (1 - \gamma \epsilon )r/\gamma$ can be rearranged to
obtain $\gamma (\epsilon r + \|\g{\delta}\|) \le r$, which implies that
$\|\Phi (\g{\theta}) - \g{\theta}^*\| \le r$ by (\ref{abs1}).
Since $\Phi$ maps $B_r(\g{\theta}^*)$ into itself and $\Phi$ is a contraction
on $B_r(\g{\theta}^*)$, the contraction mapping principle yields the
existence of a unique fixed point $\g{\theta} \in B_r(\g{\theta}^*)$.
Since $\|\Phi (\g{\theta}) - \g{\theta}^*\| = \|\g{\theta} - \g{\theta}^*\|$
for this fixed point, (\ref{abs}) is a consequence of (\ref{abs1}).
\end{proof}
\smallskip

We use Proposition~\ref{prop} with
$\g{\theta}^* = (\m{X}^*, \m{U}^*, \g{\Lambda}^*)$ defined in
(\ref{optdiscrete}) and
$\g{\theta} = (\m{X}^N, \m{U}^N, \g{\Lambda}^N)$.
The norm on $\C{X}$ is given by
\begin{equation}\label{Xnorm}
\|\g \theta\|=\|({\bf X},{\bf U}, {\g \Lambda})\|_\infty
=\max\{\|\bf X\|_\infty, \|\bf U\|_\infty,\|\g \Lambda\|_\infty\}.
\end{equation}
%
The space $\C{Y}$ corresponds to the codomain of $\C{T}$.
If $\m{y} \in \C{Y}$, then we let $\m{y}_l$ denote the part of $\m{y}$
associated with $\C{T}_l$, $1 \le l \le 6$.
The norm of $\m{y} \in \C{Y}$ is given by
\[
\|\m{y}\|_{\C{Y}} = \|\m{y}_1\|_\omega + |\m{y}_2|
+ \|\m{y}_3\|_\omega + |\m{y}_4| + h^{1/2}|\m{y}_5| +
h^{-1/2}\|\m{y}_6\|_\infty ,
\]
where for $\m{z} \in \mathbb{R}^{KNn}$, the $\omega$-norm is defined by
\[
\|\m{z}\|_\omega = \left( \sum_{k = 1}^K
\sum_{i = 1}^N \omega_i |\m{z}_{ki}|^2
\right)^{1/2}, \quad \m{z}_{ki} \in \mathbb{R}^n.
\]
For $\m{z} \in \mathbb{R}^{Nn}$, the $\omega$-norm is
\[
\|\m{z}\|_\omega = \left(
\sum_{i = 1}^N \omega_i |\m{z}_{i}|^2
\right)^{1/2}, \quad \m{z}_{i} \in \mathbb{R}^n.
\]
%

%
\section{Interpolation error in $\C{H}^1$}
\label{sect_interp}

Our estimate for the distance from 
$\C{T}(\m{X}^*, \m{U}^*, \g{\Lambda}^*)$ to $\C{F}(\m{U}^*)$
utilizes the following bound for the 
$\C{H}^1(\Omega)$ error of the interpolant based on
the point set $\tau_i$, $0 \le i \le N$.
\smallskip
\begin{lemma}\label{interp}
If $u \in \C{H}^\eta(\Omega)$ for some $\eta \ge 2$,
then there exists a constant $c$,
independent of $N$ and $\eta$, such that
\begin{equation}\label{interperror}
|u - u^I|_{\C{H}^1(\Omega)} \le
(c/N)^{p - 1} |u|_{\C{H}^p(\Omega )},
\quad p = \min \{ \eta, N+1 \},
\end{equation}
where $u^I \in \C{P}_N$ is the interpolant of $u$
satisfying $u^I(\tau_i) = u(\tau_i)$, $0 \le i \le N$, and $N > 0$.
In the case $\eta = p = 1$,
there exists a constant $c$, independent of $N$, such that
\begin{equation}\label{interperror2}
|u - u^I|_{\C{H}^1(\Omega)} \le
c |u|_{\C{H}^1(\Omega )},
\end{equation}
\end{lemma}
\smallskip

\begin{proof}
Throughout the analysis, $c$ denotes a generic constant whose value is
independent of $N$ and $\eta$, and which may have different values in
different equations.
We first show that if the lemma holds for all
$u \in \C{H}_0^1(\Omega) \cap \C{H}^\eta(\Omega)$,
then it holds for all $u \in \C{H}^\eta(\Omega)$.
Suppose $u \in \C{H}^\eta(\Omega)$ and let
$\ell$ denote the linear function for which $\ell(\pm 1) = u(\pm 1)$.
Since $\ell^I = \ell$, it follows that
\[
|u - u^I|_{\C{H}^1(\Omega)} = |(u -\ell) - (u - \ell)^I|_{\C{H}^1(\Omega)} .
\]
Since $u -\ell \in \C{H}_0^1(\Omega)$, (\ref{interperror}) gives
\[
|u - u^I|_{\C{H}^1(\Omega)} \le
(c/N)^{p - 1} |u - \ell|_{\C{H}^p(\Omega )}
\]
when $\eta \ge 2$.
Moreover, when $\eta \ge 2$,
$|u - \ell|_{\C{H}^p(\Omega )} = |u|_{\C{H}^p(\Omega )}$
since derivatives of order two or larger applied to the linear function
$\ell$ are zero.
This establishes (\ref{interperror}) for all $u \in \C{H}^\eta(\Omega)$ with
$\eta \ge 2$.
If $\eta = 1$, then by (\ref{interperror2}), we have
\begin{equation}\label{interperror3}
|u - u^I|_{\C{H}^1(\Omega)} \le
c |u - \ell|_{\C{H}^1(\Omega )} \le c\left( |u|_{\C{H}^1(\Omega)} + 
(|\ell|_{\C{H}^1(\Omega)} \right) .
\end{equation}
Since $\dot{\ell}= (u(1) - u (-1))/2$, the Schwarz inequality gives
\begin{equation}\label{interperror4}
|\ell|_{\C{H}^1(\Omega)} = \frac{|u(1) - u(-1)|}{\sqrt{2}} = 
\frac{1}{\sqrt{2}} \left| \int_{-1}^1 \dot{u}(\tau) \; d\tau \right| \le
|u|_{\C{H}^1(\Omega)} .
\end{equation}
Combine (\ref{interperror3}) and (\ref{interperror4}) to obtain
(\ref{interperror2}) for all $u \in \C{H}^1(\Omega)$.
Henceforth, it is assumed that
$u \in \C{H}_0^1(\Omega) \cap \C{H}^\eta(\Omega)$.

Let $\pi_N u$ denote the projection of $u$
into $\C{P}_N^0$ relative to the norm $| \cdot |_{\C{H}^1(\Omega)}$.
Define $E_N = u - \pi_N u$ and
$e_N = E_N^I = (u - \pi_N u)^I = u^I - \pi_N u$.
Since $E_N - e_N = u - u^I$, it follows that
\begin{equation}\label{part0}
| u - u^I|_{\C{H}^1(\Omega)} \le
|E_N|_{\C{H}^1(\Omega)} + |e_N|_{\C{H}^1(\Omega)} .
\end{equation}
In \cite[Prop.~3.1]{Elschner93} it is shown that for $\eta \ge 1$,
\begin{equation}\label{part1}
|E_N|_{\C{H}^1(\Omega)} \le (c/N)^{p-1} |u|_{\C{H}^p(\Omega)},
\quad \mbox{where } p = \min\{\eta, N+1\}.
\end{equation}
We will establish the bound
\begin{equation}\label{part2}
|e_N|_{\C{H}^1(\Omega)}  \le c |E_N|_{\C{H}^1(\Omega)}.
\end{equation}
Combine (\ref{part0})--(\ref{part2}) to obtain (\ref{interperror})
and (\ref{interperror2}) for an appropriate choice of $c$.

By \cite[Lem.~4.4]{BernardiMaday92} and the fact that
$e_N \in \C{P}_N^0$, it follows that
\begin{equation}\label{i1}
|e_N|_{\C{H}^1(\Omega)} \le cN \left( \int_\Omega
\frac{e_N^2(\tau)}{1-\tau^2} \;d\tau \right)^{1/2} .
\end{equation}
Since $e_N \in \C{P}_N^0$ and
$e_N^2(\tau)/(1-\tau^2) \in \C{P}_{2N-2}^0$,
$N$-point Radau quadrature is exact, and we have
\begin{equation}\label{i2}
\left( \int_\Omega \frac{e_N^2(\tau)}{1-\tau^2} \;d\tau
\right)^{1/2} =
\left(\sum_{i=1}^{N-1}
\frac{\omega_i e_N^2(\tau_i)}{1-\tau_i^2} \right)^{1/2} =
\left(\sum_{i=1}^{N-1}
\frac{\omega_i E_N^2(\tau_i)}{1-\tau_i^2} \right)^{1/2} .
\end{equation}
The last equality holds since $e_N = E_N$ at
$\tau_i$, $0 \le i \le N$.
Although Lemma~4.3 in \cite{BernardiMaday92} was given for Lobatto
quadrature, exactly the same proof can be used for both Gauss and
Radau quadrature.
Consequently, since $E_N \in \C{H}_0^1(\Omega)$, it follows from
\cite[Lem.~4.3]{BernardiMaday92} that
\begin{equation}\label{quad}
\left(\sum_{i=1}^N
\frac{\omega_i E_N^2(\tau_i)}{1-\tau_i^2} \right)^{1/2} \le
c\left[ \left( \int_\Omega \frac{E_N^2(\tau)}{1-\tau^2} \; d\tau\right)^{1/2}
+ N^{-1} |E_N|_{\C{H}^1(\Omega)} \right] .
\end{equation}
By \cite[Prop.~9.1]{HagerMohapatraRao16b}, we have
\begin{equation}\label{A}
N \left[ \left( \int_\Omega \frac{E_N^2(\tau)}{1-\tau^2} \;
d\tau\right)^{1/2}
+ N^{-1} |E_N|_{\C{H}^1(\Omega)} \right] \le 2|E_N|_{\C{H}^1(\Omega)}.
\end{equation}
Combine (\ref{i1}--\ref{A}) to obtain (\ref{part2}).
\end{proof}

\begin{remark}
In the analogue of Lemma~$\ref{interp}$ for the Gauss quadrature points
given in $\cite[Lem.~4.1]{HagerMohapatraRao16b}$, the exponent in the
error bound is $p - 1.5$ instead of $p - 1$.
The difference in the exponent is due to the treatment of endpoints.
In the Radau result, the polynomial interpolates
at both $\tau = -1$ and $\tau = 1$,
while in the Gauss result, the polynomial interpolates only at $\tau = -1$.
\end{remark}

\section{Analysis of the residual}
\label{residual}
The distance from
$\C{T}(\m{X}^*, \m{U}^*, \g{\Lambda}^*)$ to $\C{F}(\m{U}^*)$ is now
estimated.
\smallskip

\begin{lemma}\label{residuallemma}
If $\m{x}^*$ and $\g{\lambda}^* \in \C{PH}^\eta (\Omega_0)$
for some $\eta \ge 2$,
then there exists a constant $c$, independent of $N$, $h$, and $\eta$, such that
\begin{eqnarray}
&{\rm dist}[\C{T}(\m{X}^*, \m{U}^*, \g{\Lambda}^*), \C{F}(\m{U}^*)]_{\C{Y}}&
\nonumber \\[.1in]
&\le
h^{p-1/2}\left(\frac{c}{N}\right)^{p-1}
|{\mathbf{x}^{*}}|_{\C{PH}^p(\Omega_0)}
+ h^{q-1/2}\left(\frac{c}{N}\right)^{q-1.5}
|{\lmb^{*}}|_{\C{PH}^q(\Omega_0)},& \label{delta}
\end{eqnarray}
where $p=\min(\eta, N+1)$ and $q=\min(\eta, N)$.
\end{lemma}
\smallskip
\begin{proof}
Since $\C{T}(\m{X}^*,\m{U}^*,\g{\Lambda}^*)$ appears throughout the
analysis, it is abbreviated $\C{T}^*$.
Since the minimum principle (\ref{controlmin}) holds for all $t \in \Omega_0$,
it holds at the collocation points, which implies that
$\C{T}_6^* \in \C{F}_6(\m{U}^*)$.
Also, $\C{T}_2^* = \C{T}_5^* = \m{0}$ since the optimal state is continuous and
it satisfies the terminal condition (\ref{costate}) in the costate equation.
Thus we only need to analyze $\C{T}_1^*$, $\C{T}_3^*$, and $\C{T}_4^*$.

Let us first consider $\mathcal{T}_1^*$. 
Since ${\bf D}$ is a differentiation matrix associated with the
collocation points, we have
\begin{equation}\label{t1}
\sum_{j=0}^{N}D_{ij}
{\bf X}_{kj}^*=\dot{\bf x}_k^I(\tau_i), \quad 1\leq i \leq N,
\end{equation}
where ${\bf x}_k^I \in \mathcal{P}_N^n$ is the (interpolating) polynomial
that passes through ${\bf x}_k^*(\tau_j)$ for $0 \leq j \leq N$.
Since ${\bf x}^*$ satisfies the dynamics of (\ref{P}),
\begin{equation}\label{t1dotx}
h{\bf f}({\bf X}_{ki}^*, {\bf U}_{ki}^*)=\dot{\bf x}_k^*(\tau_i).
\end{equation}
Combine \eqref{t1} and \eqref{t1dotx} to obtain
\begin{equation}\label{t1dif}
\mathcal{T}_{1ki}^* = 
\dot{\bf x}_k^I(\tau_i)-\dot{\bf x}_k^*(\tau_i) =
\dot{\bf x}_k^I(\tau_i)-(\dot{\bf x}_k^*)^J(\tau_i),
\end{equation}
where $(\dot{\m{x}}_k^*)^J \in \C{P}_{N-1}^n$
is the interpolant that passes through
$\dot{\bf x}_k^*(\tau_i)$ for $1 \le i \le N$.
Since both $\dot{\bf x}^I$ and $(\dot{\m{x}}^*)^J$ are polynomials of
degree $N-1$ and Radau quadrature is exact for polynomials of degree
$2N -2$, it follows that
\begin{eqnarray}
\|\C{T}_1^*\|_\omega^2 &=&
\sum_{k=1}^K \sum_{i=1}^N
\omega_i | \dot{\bf x}_k^I(\tau_i)-(\dot{\bf x}_k^*)^J(\tau_i)|^2 \nonumber \\
&=&
\sum_{k=1}^K \int_{-1}^1
| \dot{\bf x}_k^I(\tau)-(\dot{\bf x}_k^*)^J(\tau)|^2 \; d \tau \nonumber \\
&\le&
2 \sum_{k=1}^K \int_{-1}^1
\left( | \dot{\bf x}_k^I(\tau)-\dot{\bf x}_k^*(\tau)|^2 +
| \dot{\bf x}_k^*(\tau)-(\dot{\bf x}_k^*)^J(\tau)|^2 \right) \; d \tau .
\label{int}
\end{eqnarray}
By Lemma~\ref{interp}, we have
\begin{equation}\label{xk-bound}
\| \dot{\bf x}_k^I-\dot{\bf x}_k^* \|_{\C{L}^2(\Omega)} \le
(c/N)^{p-1} |\m{x}^*_k|_{\C{H}^p(\Omega)}, \quad p = \min\{\eta, N+1\}.
\end{equation}
The second term in (\ref{int}) involves the difference between
between $\dot{\m{x}}_k^* \in \C{H}^{(\eta-1)}$
and its interpolant $(\dot{\bf x}_k^*)^J \in \C{P}_{N-1}^n$
at the $N$ Radau points.
By the bound given in \cite[(5.4.33)]{Canuto06} for the $\C{L}^2$ error in
Radau interpolation, this term has exactly the same bound as that
on the right side of (\ref{xk-bound}).
Since $\m{x}_k (\tau) = \m{x}(t_{k-1/2} + h\tau)$, the derivatives
contained in the right side of (\ref{xk-bound}) satisfy
\[
\frac{d^{p}\m{x}_k^*(\tau)}{d\tau^{p}} =h^{p}
\left. \frac{d^{p}\m{x}^*(t)}{dt^{p}} \right|_{t = t_{k-1/2} + h\tau}.
\]
Consequently, after a change of variables, we have
\[
\int_{-1}^1 \left| \frac{d^p\m{x}_k^*(\tau)}{d\tau^p} \right|^2 \; d\tau =
h^{2p-1} \int_{t_{k-1}}^{t_k}
\left| \frac{d^p\m{x}^*(t)}{dt^p} \right|^2 \; dt .
\]
Combine this with (\ref{int}) and (\ref{xk-bound}) to deduce that
$\|\C{T}^*_1\|_\omega$ is bounded by the first term on the right side side
of (\ref{delta}).

The analysis of $\C{T}_3^*$ is similar to the analysis of $\C{T}_1^*$.
Let $\g{\lambda}_k^I \in \C{P}_{N-1}^n$ be the polynomial
that interpolates $\g{\lambda}_k^*(\tau_j)$ for $1 \le j \le N$.
By (\ref{h-2}) and (\ref{h-1}), we have
\begin{eqnarray}
\sum_{j=1}^{N} D_{ij}^\ddagger \g{\Lambda}_{kj}^* &=&
\dot{\g{\lambda}}_k^I(\tau_i), \quad 1 \le i < N, \label{h83a} \\
\sum_{j=1}^{N} D_{Nj}^\ddagger \g{\Lambda}_{kj}^* &=&
\dot{\g{\lambda}}_k^I(\tau_i) - \g{\lambda}_k^* (1)/\omega_N. \label{h83b}
\end{eqnarray}
Since $\g{\lambda}^*$ satisfies (\ref{costate}), it follows that
\begin{equation} \label{h84}
h\nabla_x H({\bf X}_{ki}^*,{\bf U}_{ki}^*, {\g \Lambda}_{ki}^*) =
h\nabla_x H({\bf x}_k^*(\tau_i),{\bf u}_k^*(\tau_i), {\g \lambda}_k^*(\tau_i))
= -\dot{\g{\lambda}}_k^*(\tau_i) ,
\end{equation}
$1 \le i \le N$.
We substitute (\ref{h83a})--(\ref{h84}) in the definition of $\C{T}_3$ to obtain
\begin{equation}
\C{T}_{3ki}(\m{X}^*,  \m{U}^*, \g{\Lambda}^*) =
\dot{\g{\lambda}}_k^I(\tau_i) - \dot{\g{\lambda}}_k^*(\tau_i) =
\dot{\g{\lambda}}_k^I(\tau_i) - (\dot{\g{\lambda}}_k^*)^J(\tau_i) ,
\quad 1 \le i \le N,\nonumber
\end{equation}
where $(\dot{\g{\lambda}}_k^*)^J \in \C{P}_{N-1}^n$ is the polynomial that
passes through $\dot{\g{\lambda}}_k^*(\tau_i)$, $1 \le i \le N$.
Note that the term $-\g{\lambda}_k^* (1)/\omega_N$ in (\ref{h83b})
cancels the corresponding term in $\C{T}_{3k}$ due to the
continuity of $\g{\lambda}^*$.
Since $\dot{\g{\lambda}}_k^I \in \C{P}_{N-2}^n$ and
$(\dot{\g{\lambda}}_k^*)^J \in \C{P}_{N-1}^n$, and since
Radau quadrature is exact for polynomials of degree $2N-2$,
we obtain, as in (\ref{int}),
\begin{equation}\label{T4bound}
\|\C{T}_3^*\|_\omega^2 \le
2 \sum_{k=1}^K \int_{-1}^1
\left( | \dot{\g{\lambda}}_k^I(\tau)-\dot{\g{\lambda}}_k^*(\tau)|^2 +
| \dot{\g{\lambda}}_k^*(\tau)-(\dot{\g{\lambda}}_k^*)^J(\tau)|^2 \right)
\; d \tau .
\end{equation}
The last term in (\ref{T4bound}) has the bound
\begin{equation}\label{lambdak-final}
\| (\dot{\g{\lambda}}^*_k)^J-\dot{\g{\lambda}}_k^* \|_{\C{L}^2(\Omega)} \le
h^p (c/N)^{p-1} |\g{\lambda}^*|_{\C{H}^p(t_{k-1}, t_k)},
\quad p = \min\{\eta, N+1\},
\end{equation}
corresponding to the $\C{L}^2$ error in interpolation at the Radau points.
The other term, however, is different from the state
since $\g{\lambda}_k^I$ has degree $N-1$ while the state
$\m{x}_k^I$ has degree $N$, and the state interpolates at both
the quadrature points and at $\tau = -1$, while $\g{\lambda}_k^I$
only interpolates at the quadrature points.
The error in the derivative of the interpolant at the Radau
points has the bound \cite[(5.4.34)]{Canuto06}
\begin{equation}\label{lambdak-final2}
\| \dot{\g{\lambda}}_k^I-\dot{\g{\lambda}}_k^* \|_{\C{L}^2(\Omega)} \le
h^q (c/N)^{q-1.5} |\g{\lambda}^*|_{\C{H}^q(\Omega)},
\quad q = \min\{\eta, N\} .
\end{equation}
The exponent changes from $p-1$ in (\ref{lambdak-final}) to $q-1.5$ due to the
fact that $\g{\lambda}_k^I$ does not interpolate at $\tau = -1$,
and $q\le p$ since the polynomial associated with
$\g{\lambda}_k^I$ has degree $N-1$.
Note that if $\g{\lambda}^* \in \C{PH}^\eta(\Omega_0)$,
then $\g{\lambda}^* \in \C{PH}^{(\eta-1)}(\Omega_0)$,
so we can always ensure that the error bound (\ref{lambdak-final2}) dominates
the error bound (\ref{lambdak-final}) by lowering $\eta$ in
(\ref{lambdak-final}) if necessary.
Utilizing the bound (\ref{lambdak-final2}) in (\ref{T4bound}) and changing
variables from $\tau$ to $t$, we deduce that $\|\C{T}_3\|_\omega$ is
bounded by the second term on the right side of (\ref{delta}).

Finally, let us consider $\C{T}_4^*$.
Applying (\ref{h84}) and utilizing the continuity of $\g{\lambda}^*$ and
the exactness of Radau quadrature, we have
\begin{eqnarray*}
\C{T}_{4k}^* &=&
{\g \lambda}_k^*(-1)-{\g \lambda}_{k+1}^*(-1)
+\sum_{i=1}^{N}\omega_i \dot{\g{\lambda}}_k^* (\tau_i) \\
&=&
{\g \lambda}_k^*(-1)-{\g \lambda}_{k}^*(1)
+\sum_{i=1}^{N}\omega_i (\dot{\g{\lambda}}_k^*)^J (\tau_i) \\
&=&
{\g \lambda}_k^*(-1)-{\g \lambda}_{k}^*(1) +
\int_{-1}^1 (\dot{\g{\lambda}}_k^*)^J (\tau) \; d\tau
= \int_{-1}^1 [(\dot{\g{\lambda}}_k^*)^J (\tau)
- \dot{\g{\lambda}}_k^*(\tau)] \; d\tau .
\end{eqnarray*}
By (\ref{lambdak-final}) and the Schwarz inequality, we have
\[
|\C{T}_{4k}^*| \le \sqrt{2}
\|(\dot{\g{\lambda}}_k^*)^J - \dot{\g{\lambda}}_k^* \|_{\C{L}^2(\Omega)} \le
h^p (c/N)^{p-1} |\g{\lambda}^*|_{\C{H}^p(\Omega)},
\quad p = \min\{\eta, N+1\}.
\]
As in the analysis of $\C{T}_3$,
we square this, sum over $k$, change variables from $\tau$ to $t$,
and take the square root to obtain a bound
that can be dominated by the last term in (\ref{delta}).
This completes the proof.
\end{proof}
%

\section{Invertibility of linearized dynamics}
\label{inverse}
The inclusion
\[
\C{T}(\m{X}, \m{U}, \g{\Lambda}) \in \C{F}(\m{U}),
\]
corresponding to the first-order optimality conditions for the discrete
problem (\ref{D}), will be linearized around $(\m{X}^*, \m{U}^*,\g{\Lambda}^*)$.
Given $\m{Y} \in \C{Y}$, the linearized problem is to find
$(\m{X}, \m{U}, \g{\Lambda})$ such that
\begin{equation}\label{linearproblem}
(\nabla \C{T}^*) [\m{X}, \m{U}, \g{\Lambda}] +
\m{Y} \in \C{F}(\m{U}),
\end{equation}
where $\nabla \C{T}^*$ denotes $\nabla \C{T}(\m{X}^*, \m{U}^*, \g{\Lambda}^*)$,
the derivative of $\C{T}$ evaluated at $(\m{X}^*, \m{U}^*, \g{\Lambda}^*)$.
Since $\g{\Lambda}$ enters $\C{T}$ in an affine manner, the linearization
with respect to $\g{\Lambda}$ is trivial.
On the other hand,
the discrete state $\m{X}$ and the discrete control $\m{U}$
generally enter $\C{T}$ in a nonlinear fashion.
The derivative of $\C{T}$ in (\ref{T1})--(\ref{T6})
is built from the following matrices for $1 \le k \le K$:
\[
\begin{array}{ll}
{\bf A}_{ki}=\nabla_x{\bf f}({\bf x}^*(t_{ki}),{\bf u}^*(t_{ki})),
&{\bf B}_{ki}=\nabla_u{\bf f}({\bf x}^*(t_{ki}),{\bf u}^*(t_{ki})),\\
{\bf Q}_{ki}=\nabla_{xx}^2H\left({\bf x}^*(t_{ki}),{\bf u}^*(t_{ki}),
{\g \lambda}^*(t_{ki})\right),
&{\bf S}_{ki}=\nabla_{xu}^2H\left({\bf x}^*(t_{ki}),{\bf u}^*(t_{ki}),
{\g \lambda}^*(t_{ki})\right),\\
{\bf R}_{ki}=\nabla_{uu}^2H\left({\bf x}^*(t_{ki}),{\bf u}^*(t_{ki}),
{\g \lambda}^*(t_{ki})\right), &{\bf T}=\nabla^2C({\bf x}^*(1)).
\end{array}
\]
As pointed out in (\ref{optdiscrete}), the optimal variables
$(\m{x}^*, \m{u}^*, \g{\lambda}^*)$ evaluated at the $t_{ki}$
are equivalent to the transformed optimal variables
$(\m{x}_k^*, \m{u}_k^*, \g{\lambda}_k^*)$ evaluated at the $\tau_i$.
The elements of $\nabla \C{T}^*[\m{X}, \m{U}, \g{\Lambda}]$ are the
following:
\begin{eqnarray}
\nabla \C{T}_{1ki}^*[\m{X}, \m{U}, \g{\Lambda}] &=&
\left( \sum_{j=0}^{N}{D}_{ij}{\bf X}_{kj} \right)
- h(\m{A}_{ki} \m{X}_{ki} + \m{B}_{ki} \m{U}_{ki}),
\quad 1\leq i\leq N ,\nonumber \\
\nabla \C{T}_{2k}^*[\m{X}, \m{U}, \g{\Lambda}] &=&
{\bf X}_{k0} - \m{X}_{k-1, N}, \quad \mbox{where }
\m{X}_{0N} = \m{0},
\nonumber \\
\nabla \C{T}_{3ki}^*[\m{X}, \m{U}, \g{\Lambda}] &=&
\left( \sum_{j=1}^{N}{D}_{ij}^\ddagger{\g \Lambda}_{kj} \right) +
h (\m{A}_{ki}\tr \g{\Lambda}_{ki}  + \m{Q}_{ki} \m{X}_{ki}
+ \m{S}_{ki} \m{U}_{ki}),
\quad 1 \leq i < N , \nonumber \\
\nabla \C{T}_{3kN}^*[\m{X}, \m{U}, \g{\Lambda}] &=&
\left( \sum_{j=1}^{N}{D}_{Nj}^\ddagger{\g \Lambda}_{kj} \right) +
h (\m{A}_{kN}\tr \g{\Lambda}_{kN}  + \m{Q}_{kN} \m{X}_{kN}
+ \m{S}_{kN} \m{U}_{kN}) \nonumber \\
&& \quad \quad \quad \quad \quad \quad \quad \;
+ \g{\Lambda}_{k+1,0}/\omega_N , \\[.05in]
\nabla \C{T}_{4k}^*[\m{X}, \m{U}, \g{\Lambda}] &=&
\g{\Lambda}_{k0} - \g{\Lambda}_{k+1,0} - h \sum_{i=1}^N \omega_i
(\m{A}_{ki}\tr \g{\Lambda}_{ki}  +
\m{Q}_{ki} \m{X}_{ki} + \m{S}_{ki} \m{U}_{ki}),
\nonumber \\
\nabla \C{T}_{5k}^*[\m{X}, \m{U}, \g{\Lambda}] &=&
\m{TX}_{KN} - \g{\Lambda}_{K+1,0} , \nonumber \\[.1in]
\nabla \C{T}_{6ki}^*[\m{X}, \m{U}, \g{\Lambda}] &=&
-h(\m{B}_{ki}\tr \g{\Lambda}_{ki} + \m{S}_{ki}\tr\m{X}_{ki} +
\m{R}_{ki} \m{U}_{ki}), \quad 1\leq i\leq N , \nonumber
\end{eqnarray}
where $1 \le k \le K$.

The following result establishes invertibility of the linearized state dynamics.
\smallskip

\begin{lemma}
\label{feasiblestate}
If {\rm (A2)} holds,
then for each $\m{q}\in \mathbb{R}^{Kn}$ and
$\m{p} \in \mathbb{R}^{KNn}$
with $\m{q}_k$ and $\m{p}_{ki} \in \mathbb{R}^n$,
the linear system
\begin{eqnarray}
\sum_{j=0}^{N}{D}_{ij}{\bf X}_{kj} &=&
h\m{A}_{ki} \m{X}_{ki} + \m{p}_{ki},
\quad 1\leq i\leq N , \label{h99} \\
{\bf X}_{k0} &= & \m{X}_{k-1, N} + \m{q}_{k},
\quad \m{X}_{0N} = \m{0},
\label{h100}
\end{eqnarray}
$1 \le k \le K$,
has a unique solution $\m{X} \in \mathbb{R}^{K(N+1)n}$.
This solution has the bound
\begin{equation} \label{xjbound}
\sup_{\substack{1 \le k \le K \\[.01in]
1 \le j \le N}} \|\m{X}_{kj}\|_\infty \le
h^{-1/2} \left( \frac{\sqrt{2} \|\m{p}\|_\omega + |\m{q}|}
{(1-2hd_1)^K} \right) .
\end{equation}
\end{lemma}
\smallskip

\begin{remark}
Recall that $d_1$ is defined in $(\ref{d1d2})$.
Since the denominator expression $(1-2hd_1)^K = (1-d_1/K)^K$
in the bound $(\ref{xjbound})$ approaches $e^{-d_1}$ as $K$ tends to infinity,
the denominator is bounded away from zero, uniformly in $K$.
Hence, $(\ref{xjbound})$ also implies a uniform bound, independent of $K$.
\end{remark}
\smallskip

\begin{proof}
We first show that for given $\m{X}_{k0}$, the linear system
(\ref{h99}) uniquely determines $\m{X}_{k1}$ through $\m{X}_{kN}$.
Since $\m{X}_{0N} = \m{0}$, it follows from (\ref{h100}) that
$\m{X}_{10} = \m{q}_1$ is known.
Consequently, for $k = 1$ up to $k = K$, we can use (\ref{h99}) to compute
$\m{X}_{k1}$ through $\m{X}_{kN}$, and then (\ref{h100}) to
evaluate $\m{X}_{k+1,0}$.
This shows that (\ref{h99})--(\ref{h100}) has a unique solution that
can be computed by a recursive process.

Let $\bar{\m{X}}_k$ be the vector obtained by vertically stacking
$\m{X}_{k1}$ through $\m{X}_{kN}$,
let ${\m{A}}_k$ be the block diagonal matrix
with $i$-th diagonal block $\m{A}_{ki}$, $1 \le i \le N$,
define $\bar{\m{D}} = {\bf D}_{1:N}\otimes {\bf I}_n$ where
$\otimes$ is the Kronecker product, and let $\m{D}_0$ denote the
first column of $\m{D}$.
With this notation, (\ref{h99})--(\ref{h100}) reduce to
\begin{equation}\label{statedynamics}
(\bar{\m{D}} - h{\m{A}}_k) \bar{\m{X}}_k = \m{p} - (\m{D}_0\otimes\m{I}_n)
\m{X}_{k0} = \m{p} - (\m{D}_0\otimes\m{I}_n) (\m{X}_{k-1,N} + \m{q}_k) .
\end{equation}
By (P1), ${\bf D}_{1:N}$ is invertible and
$\|{\bf D}_{1:N}^{-1}\|_\infty = 2$.
Hence, $\|\bar{\m{D}}^{-1}\| =$
$\|{\bf D}_{1:N}^{-1}\otimes {\bf I}_n\| = 2$, and
by (A2), we have $2h\|{\m{A}}_k\|_\infty \le 2hd_1 < 1$, which implies that
\[
h\|\bar{\m{D}}^{-1}{\m{A}}_k\|_\infty \le
h\|\bar{\m{D}}^{-1}\|_\infty \|{\m{A}}_k\|_\infty \le
2 hd_1 < 1 .
\]
By \cite[p. 351]{HJ12},
${\bf I}-h\bar{\bf D}^{-1}{\bf A}_k$ is invertible and
\begin{equation}\label{Xbound}
\|({\bf I}-h\bar{\bf D}^{-1}{\bf A})^{-1}\|_\infty \leq 1/(1-2hd_1).
\end{equation}
Multiply (\ref{statedynamics}) first by
$\bar{\bf D}^{-1}$ and then by
$(\m{I} - h\bar{\m{D}}^{-1}{\m{A}}_k)^{-1}$ to obtain
\[
\bar{\m{X}}_k =
({\bf I}-h\bar{\bf D}^{-1}{\bf A}_k)^{-1}
\left( \bar{\m{D}}^{-1}\m{p}_k +
\bar{\m{D}}^{-1} (\m{D}_0\otimes \m{I}_n) (\m{X}_{k-1,N} + \m{q}_k) \right) .
\]
It is shown in \cite[Lem.~5.1]{HagerMohapatraRao16b} that
$(\bar{\m{D}})^{-1} [\m{D}_{0} \otimes \m{I}_n] =$ $-\m{1}\otimes \m{I}_n$.
Consequently,
\[
\bar{\m{X}}_k =
({\bf I}-h\bar{\bf D}^{-1}{\bf A}_k)^{-1}
\left( \bar{\m{D}}^{-1}\m{p}_k - \m{1}\otimes
(\m{X}_{k-1,N} + \m{q}_k) \right) .
\]
Take norms and apply (\ref{Xbound}) to get
\begin{equation}\label{h98}
\|\bar{\m{X}}\|_\infty \le
\left( \frac{1}{1-2hd_1} \right)
\left( \|\bar{\m{D}}^{-1}\m{p}_k\|_\infty +
\|\m{X}_{k-1,N}\|_\infty + \|\m{q}_k\|_\infty \right) .
\end{equation}
In \cite[Lem.~5.1]{HagerMohapatraRao16b} it is shown that by (P2), we have
\[
\|\bar{\m{D}}^{-1}\m{p}_k\|_\infty  \le \sqrt{2}
\|\m{p}_k\|_\omega, \quad
\|\m{p}_k\|_\omega = \left( \sum_{i = 1}^N \omega_i |\m{p}_{ki}|^2 \right) .
\]
Insert this bound in (\ref{h98}) and utilize the trivial inequality
$\|\m{q}_k\|_\infty \le |\m{q}_k|$ to obtain
\begin{equation}\label{kterm}
\|\bar{\m{X}}_k\|_\infty \le \left( \frac{1}{1-2hd_1} \right)
\left(\|{\m{X}}_{k-1,N}\|_\infty +
\sqrt{2}\|\m{p}_k\|_\omega + |\m{q}_k| \right).
\end{equation}
%
Since $\m{X}_{kN} = \m{0}$ for $k = 0$ and
$\|\m{X}_{k,N}\|_\infty \le \|\bar{\m{X}}_{k}\|_\infty$ for $k > 0$,
(\ref{kterm}) yields
\begin{equation}\label{xk}
\|\bar{\m{X}}_k\|_\infty \le
\sum_{j=1}^k \frac{\sqrt{2}\|\m{p}_j\|_\omega +
|\m{q}_j|}{(1-2hd_1)^{k-j+1}}
\end{equation}
for $1 \le k \le K$.
The upper bound (\ref{xjbound}) is obtained by
replacing $1/(1-2hd_1)^{k-j+1}$ by
its maximum $1/(1-2hd_1)^{K}$ and by
utilizing the Schwarz inequality as in
\begin{equation}\label{pandq}
\sum_{j=1}^k \|\m{p}_j\|_\omega \le \sqrt{k} \|\m{p}\|_\omega \le
h^{-1/2} \|\m{p}\|_\omega
\quad
\mbox{and} \quad
\sum_{j=1}^k |\m{q}_j| \le \sqrt{k} |\m{q}| \le h^{-1/2} |\m{q}|.
\end{equation}
\end{proof}
\smallskip

The linearized costate dynamics has an analogous bound.
\smallskip

\begin{lemma}
\label{feasiblecostate}
If {\rm (P4)} and {\rm (A2)} hold,
then for each $\m{q}\in \mathbb{R}^{Kn}$,
$\m{p} \in \mathbb{R}^{KNn}$, and $\g{\Lambda}_{K+1,0} \in \mathbb{R}^n$ with
$\m{q}_k$ and $\m{p}_{ki} \in \mathbb{R}^n$,
the linear system
\begin{eqnarray}
\sum_{j=1}^{N}{D}_{ij}^\ddagger{\g \Lambda}_{kj} &=&
\m{p}_{ki} - h\m{A}_{ki}\tr \g{\Lambda}_{ki},
\quad 1 \leq i < N , \label{cos1} \\
\sum_{j=1}^{N}{D}_{Nj}^\ddagger{\g \Lambda}_{kj} &=&
\m{p}_{kN} - h\m{A}_{kN}\tr \g{\Lambda}_{kN}
- \g{\Lambda}_{k+1,0}/\omega_N , \label{cos2} \\
\g{\Lambda}_{k0} &=& \g{\Lambda}_{k+1,0} + \m{q}_{k}
+ h \sum_{i=1}^N \omega_i \m{A}_{ki}\tr \g{\Lambda}_{ki} , \label{cos3}
\end{eqnarray}
$1 \le k \le K$,
has a unique solution $\g{\Lambda} \in \mathbb{R}^{K(N+1)n}$.
This solution has the bound
\begin{equation} \label{cjbound}
\|\g{\Lambda}\|_\infty \le
\frac{\|\g{\Lambda}_{K+1,0}\|_\infty
+ h^{-1/2} \sqrt{2}
\|\m{p}\|_\omega + \sum_{k=1}^K |\m{q}_k|}{(1-2hd_2)^K} .
\end{equation}
\end{lemma}
\smallskip


\begin{proof}
The proof is similar to the proof of Lemma~\ref{feasiblestate} except
that the recursive solution of (\ref{cos1})--(\ref{cos3})
starts from $k = K$ and descends to $k = 1$.
In particular, we first show that for given $\g{\Lambda}_{k+1,0}$, the
linear system (\ref{cos1})--(\ref{cos2}) uniquely determines
$\g{\Lambda}_{k1}$ through $\g{\Lambda}_{kN}$;
then (\ref{cos3}) can be used to evaluate $\g{\Lambda}_{k0}$.

Define $\bar{\m{D}}^\ddag = {\bf D}^\ddag \otimes {\bf I}_n$,
where $\otimes$ is the Kronecker product.
Equations (\ref{cos1}) and (\ref{cos2}) can be combined into the single
equation
\begin{equation}\label{combined}
\bar{\m{D}}^\ddag \bar{\g{\Lambda}}_k =
\m{p}_{k} - h\m{A}_{k}\tr \bar{\g{\Lambda}}_{k}
- (\m{e}_N\otimes \m{I}_n) 
\g{\Lambda}_{k+1,0}/\omega_N,
\end{equation}
where
$\bar{\g{\Lambda}}_k$ is obtained by vertically stacking
$\g{\Lambda}_{k1}$ through $\g{\Lambda}_{kN}$ and
$\m{e}_N$ is the vector whose $N$ components are all zero except
for the last component which is 1.
By (\ref{h-2}) and (\ref{h-1}), $\m{D}^\ddag \m{1} = -\m{e}_N/\omega_N$,
which implies that
\begin{equation}\label{h999}
\m{D}^{\ddag\; -1} \m{e}_N = -\omega_N \m{1}.
\end{equation}
Hence, we have
\[
\bar{\m{D}}^{\ddag\;-1} (\m{e}_N\otimes \m{I}_n)/\omega_N =
[{\bf D}^{\ddag \; -1}\otimes {\bf I}_n]
(\m{e}_N\otimes \m{I}_n)/\omega_N =
-\m{1}\otimes \m{I}_n .
\]
Multiply (\ref{combined}) by $\bar{\m{D}}^{\ddag\;-1}$ and rearrange to obtain
\begin{equation}\label{rearrange}
(\m{I} + h\bar{\m{D}}^{\ddag\;-1}\m{A}_k\tr)\bar{\g{\Lambda}}_k = 
\bar{\m{D}}^{\ddag\;-1}\m{p}_k + (\m{1} \otimes \m{I}_n)\g{\Lambda}_{k+1,0} .
\end{equation}
As noted in Section~\ref{introduction}, (P4) implies that (P3) holds;
that is, $\|\bar{\bf D}^{\ddag \; -1}\|_\infty \le 2$.
By (A2),
$h\|\bar{\m{D}}^{\ddag\;-1}\m{A}_k\tr \|_\infty \le 2hd_2 < 1$.
Consequently, the matrix
${\bf I}+h\bar{\bf D}^{\ddag \; -1}{\bf A}_k\tr$ is invertible with
\[
\left\|\left({\bf I}+h\bar{\bf D}^{\ddag \; -1}
{\bf A}_k\tr\right)^{-1}\right\|_\infty 
\le \frac{1}{1-2hd_2} .
\]
Multiply (\ref{rearrange}) by
$({\bf I}+h\bar{\bf D}^{\ddag \; -1}{\bf A}_k\tr)^{-1}$ and take
the norm of each side to obtain
\begin{eqnarray}
\|\bar{\g{\Lambda}}_k\|_\infty &\le& \left( \frac{1}{1-2hd} \right)
(\|{\g{\Lambda}}_{k+1,0}\|_\infty +
\|\bar{\m{D}}^{\ddag \; -1}\m{p}_k\|_\infty) \label{pre-cterm} \\
&\le&
\left( \frac{1}{1-2hd} \right)
(\|{\g{\Lambda}}_{k+1,0}\|_\infty +
\|\bar{\m{D}}^{\ddag \; -1}\m{p}_k\|_\infty + \|\m{q}_k\|_\infty )
\label{cterm}
\end{eqnarray}
The norm of (\ref{cos3}) gives
\begin{eqnarray*}
\|\g{\Lambda}_{k0}\|_\infty &\le& \|\g{\Lambda}_{k+1,0}\|_\infty +
\|\m{q}_{k}\|_\infty +
h \sum_{i=1}^N \omega_i \|\m{A}_{ki}\tr\|_\infty
\|\g{\Lambda}_{ki}\|_\infty \\
&\le& \|\g{\Lambda}_{k+1,0}\|_\infty + \|\m{q}_{k}\|_\infty +
2hd_2\|\bar{\g{\Lambda}}_k\|_\infty
\end{eqnarray*}
since the $\omega_i$ sum to 2.
Using the bound for $\|\bar{\g{\Lambda}}_k\|_\infty$ from
(\ref{pre-cterm}) and the fact that $2hd_2 < 1$, we have
\begin{eqnarray}
\|\g{\Lambda}_{k0}\|_\infty &\le& \|\m{q}_k\|_\infty +
\left( \frac{1}{1-2hd_2} \right)
(\|\g{\Lambda}_{k+1,0}\|_\infty +
2hd_2 \|\bar{\m{D}}^{\ddag \; -1}\m{p}_k\|_\infty) \nonumber \\
&\le&
\left( \frac{1}{1-2hd_2} \right)
(\|\g{\Lambda}_{k+1,0}\|_\infty +
2hd_2 \|\bar{\m{D}}^{\ddag \; -1}\m{p}_k\|_\infty +
\|\m{q}_k\|_\infty) \nonumber \\
&\le&
\left( \frac{1}{1-2hd_2} \right)
(\|\g{\Lambda}_{k+1,0}\|_\infty +
\|\bar{\m{D}}^{\ddag \; -1}\m{p}_k\|_\infty +
\|\m{q}_k\|_\infty). \label{0bound}
\end{eqnarray}
Since $\g{\Lambda}_{k,0}$ is contained in $\g{\Lambda}_k$, it follows that
$\|\g{\Lambda}_{k,0}\|_\infty \le$ $\|\g{\Lambda}_{k}\|_\infty$.
Combine (\ref{cterm}) and (\ref{0bound}) to obtain
\[
\|\g{\Lambda}_{k}\|_\infty \le
\left( \frac{1}{1-2hd_2} \right) (\|\g{\Lambda}_{k+1}\|_\infty +
\|\bar{\m{D}}^{\ddag \; -1}\m{p}_k\|_\infty + \|\m{q}_k\|_\infty) ,
\]
where we define $\g{\Lambda}_{K+1,j} = \m{0}$ for $j > 0$ so that
$\|\g{\Lambda}_{K+1}\|_\infty =$ $\|\g{\Lambda}_{K+1,0}\|_\infty$.
This inequality is applied recursively to obtain
\[
\|\g{\Lambda}_{k}\|_\infty \le
\frac{\|\g{\Lambda}_{K+1}\|_\infty}{(1-2hd_2)^{K+1-k}}
+ \sum_{j=k}^K \left( \frac{
\|\bar{\m{D}}^{\ddag \; -1}\m{p}_j\|_\infty + \|\m{q}_j\|_\infty}
{(1-2hd_2)^{j-k+1}} \right) .
\]
To bound the right side, 
the factors $1/(1-2hd_2)^{j-k+1}$ are replaced by
their maximum $1/(1-2hd_2)^{K}$ to obtain
\[
\|\g{\Lambda}_{k}\|_\infty \le
\frac{\|\g{\Lambda}_{K+1}\|_\infty + \sum_{j=k}^K
\left[ \|\bar{\m{D}}^{\ddag \; -1}\m{p}_j\|_\infty + |\m{q}_j| \right] }
{(1-2hd_2)^K} .
\]
By the analysis given in \cite[Lem.~5.1]{HagerMohapatraRao16b},
(P4) implies that
$\|\bar{\m{D}}^{\ddag \; -1}\m{p}_j\|_\infty \le \sqrt{2}\|\m{p}_j\|_\omega$.
Hence, we have
\[
\|\g{\Lambda}_{k}\|_\infty \le
\frac{\|\g{\Lambda}_{K+1}\|_\infty + \sum_{j=k}^K
\left[ \sqrt{2} \|\m{p}_j\|_\omega + |\m{q}_j| \right] }
{(1-2hd_2)^K} .
\]
The first inequality in (\ref{pandq}) completes the proof of (\ref{cjbound}).
\end{proof}
\section{Invertibility of $\C{F} - \nabla \C{T}^*$}
\label{Invert}
Now let us consider the invertibility of $\C{F} - \nabla \C{T}^*$.

\begin{proposition}\label{invertible}
If {\rm (A1)}--{\rm (A2)} and {\rm (P4)} hold, then
for each $\m{Y} \in \C{Y}$, there is a unique solution
$(\m{X}, \m{U}, \g{\Lambda})$ to $(\ref{linearproblem})$.
\end{proposition}
\smallskip
\begin{proof}
Similar to the strategy used in
\cite{HagerDontchevPooreYang95,DontchevHager93,DontchevHager98a,
DontchevHagerVeliov00,Hager90,HagerHouRao15c,HagerHouRao16c,
HagerMohapatraRao16b},
a strongly convex quadratic programming problem is formulated;
the quadratic program is constructed so that
the first-order optimality conditions reduce to (\ref{linearproblem}).
In particular, we consider the problem

\begin{equation}\label{QP}
\left.
\begin{array}{cll}
\mbox {minimize} &\frac{1}{2} \mathcal{Q}({\bf X},{\bf U})
+ \C{L}(\m{X}, \m{U}, \m{Y}) &\\[.08in]
\mbox {subject to} &\sum_{j=0}^{N}{D}_{ij}{\bf X}_{kj} =
h({\bf A}_{ki}{\bf X}_{ki}+ {\bf B}_{ki}{\bf U}_{ki}) -{\bf y}_{1ki},
& \m{U}_{ki} \in \C{U}, \\
&\m{X}_{k0} =\m{X}_{k-1,N} - \m{y}_{2k}, \quad \m{X}_{0N} = \m{0}, &
\end{array}
\right\}
\end{equation}
where $1\leq i \leq N$ and $1 \le k \le K$.
The quadratic and linear terms in the objective are
\begin{eqnarray}
\C{Q}({\bf X},{\bf U}) &=&
{\bf X}_{KN}\tr{\bf T}{\bf X}_{KN} +
h\sum_{k=1}^K\sum_{i=1}^N \omega_i
\left[ \begin{array}{l} {\bf X}_{ki} \\ {\bf U}_{ki} \end{array} \right]\tr
\left[ \begin{array}{ll}
{\bf Q}_{ki} & {\bf S}_{ki} \\
{\bf S}_{ki}\tr & {\bf R}_{ki} \end{array}\right]
\left[ \begin{array}{l}{\bf X}_{ki} \\ {\bf U}_{ki} \end{array} \right],
\quad \quad \quad
\label{Q} \\[.05in]
\C{L}(\m{X}, \m{U}, \m{Y}) &=&
\m{y}_5\tr \m{X}_{KN} +
\sum_{k=1}^K\sum_{i=1}^N \omega_i \left( \m{y}_{3ki}\tr
\m{X}_{ki} -\m{y}_{6ki}\tr \m{U}_{ki} \right) \\
&& \quad \quad \quad \quad -\sum_{k=1}^{K}\m{X}_{k0}\tr \left( \m{y}_{4k}
 + \sum_{i=1}^N \omega_i \m{y}_{3ki} \right) . 
\label{L}
\end{eqnarray}
In (\ref{QP}), the minimization is over $\m{X}$ and $\m{U}$, while
$\m{Y}$ is a fixed parameter.
By Lemma~\ref{feasiblestate}, the quadratic programming problem (\ref{QP})
is feasible, and by the continuity condition, $\m{X}_{k0}$ can be
eliminated from (\ref{QP}).
Since the Radau quadrature weights $\omega_i$ are strictly positive,
it follows from (A1) that $\C{Q}$ is strongly convex relative to
$\m{X}_{ki}$ and $\m{U}_{ki}$, where $1 \le i \le N$ and $1 \le k \le K$.
Hence, there exists a unique optimal solution to (\ref{QP})
for any choice of $\m{Y}$.
We now show that the first-order optimality conditions for (\ref{QP}) reduce to
$\nabla \C{T}^*[\m{X}, \m{U}, \g{\Lambda}] + \m{Y} \in \C{F}(\m{U})$.
The first-order optimality conditions hold since $\C{U}$ has nonempty
interior.
Since the first-order optimality conditions are both necessary and
sufficient for optimality in this convex setting,
there exists a solution to (\ref{linearproblem}).
Uniqueness of $\m{X}$ and $\m{U}$ is due to (A1) and the strong
convexity of (\ref{QP}).
Uniqueness of $\g{\Lambda}$ is by Lemma~\ref{feasiblecostate}.

The derivation of the first-order optimality conditions for (\ref{QP})
is essentially the same process that we used in Section~\ref{abstract}
to write the first-order optimality conditions for the discrete problem
(\ref{D}) as $\C{T}(\m{X}, \m{U}, \g{\Lambda}) \in \C{F}(\m{U})$.
The first two components of
$\nabla \C{T}^*[\m{X}, \m{U}, \g{\Lambda}] + \m{Y} \in \C{F}(\m{U})$ are
simply the constraints of (\ref{QP}).
The Lagrangian $L$ for (\ref{QP}) is
\begin{eqnarray*}
L(\g{\lambda}, \m{X}, \m{U}) &=&
{\textstyle\frac{1}{2}}\C{Q}(\m{X}, \m{U}) + \C{L}(\m{X}, \m{U}, \m{Y}) 
+ \sum_{k=1}^K\left\langle \g{\lambda}_{k0}, 
\left({\bf X}_{k-1,N}- \m{y}_{2k} - {\bf X}_{k0} \right) \right\rangle \\
&&
+ \sum_{k=1}^K\sum_{i=1}^N\left\langle\g{\lambda}_{ki},
h({\bf A}_{ki}{\bf X}_{ki}+ {\bf B}_{ki}{\bf U}_{ki}) - \m{y}_{1ki}
-\sum_{j=0}^{N}{D}_{ij}{\bf X}_{kj} \right\rangle
\end{eqnarray*}
The negative derivative of the Lagrangian with respect to $\m{U}_{ki}$ is
\[
-h \left[ \m{B}_{ki}\tr \g{\lambda}_{ki} + \omega_i
(\m{S}_{ki}\tr \m{X}_{ki} + \m{R}_{ki}\m{U}_{ki})\right] + \omega_i \m{y}_{6ki}.
\]
After the substituting
$\g{\lambda}_{ki} = \omega_i \g{\Lambda}_{ki}$,
the requirement that this vector lies in $N_{\C{U}}(\m{U}_{ki})$ leads
the 6th component of (\ref{linearproblem}).
Equating to zero the derivative of the Lagrangian with respect to
$\m{X}_{kj}$, $1 \le j < N$, yields the relation
\[
\sum_{i=1}^N {D}_{ij}\g{\lambda}_{ki} =
h \left[ \m{A}_{kj}\tr\g{\lambda}_{kj} +
\omega_j (\m{Q}_{kj}\m{X}_{kj} + \m{S}_{kj}\m{U}_{kj}) \right] +
\omega_j \m{y}_{3kj} .
\]
Equating to zero the derivative of the Lagrangian with respect to
$\m{X}_{kN}$ yields the relation
\[
\sum_{i=1}^N {D}_{ij}\g{\lambda}_{ki} =
h \left[ \m{A}_{kj}\tr\g{\lambda}_{kj} +
\omega_j (\m{Q}_{kj}\m{X}_{kj} + \m{S}_{kj}\m{U}_{kj}) \right] +
\omega_j \m{y}_{3kj} + \g{\lambda}_{k+1,0},
\]
where $\g{\lambda}_{K+1, 0} = \m{TX}_{KN} + \m{y}_{5}$.
After substituting $D_{ij} = -D^\ddagger_{ji}\omega_j/\omega_i$,
$\g{\lambda}_{ki} =$ $\omega_i \g{\Lambda}_{ki}$,
and $\g{\lambda}_{k0} = \g{\Lambda}_{k0}$,
we obtain the 3rd and 5th components of (\ref{linearproblem}).

Finally, we equate to zero the derivative of the Lagrangian with
respect to $\m{X}_{k0}$:
\[
\sum_{i=1}^N{D}_{i0}{\g \lambda}_{ki} =
- \left( \g{\lambda}_{k0} + \m{y}_{4k} +
\sum_{i=1}^N \omega_i \m{y}_{3ki} \right) .
\]
Utilizing the identity (\ref{eq06}), it follows that
\begin{equation}\label{Di0}
\sum_{i=1}^N\sum_{j=1}^N \omega_i D_{ij}^\ddag\g{\Lambda}_{kj} =
- \left( \g{\Lambda}_{k0} + \m{y}_{4k} +
\sum_{i=1}^N \omega_i \m{y}_{3ki} \right) .
\end{equation}
Multiply the equations in the 3rd component of (\ref{linearproblem}) by
$\omega_i$ and sum over $i$ to obtain
\[
\sum_{i=1}^N\sum_{j=1}^N \omega_i D_{ij}^\ddag\g{\Lambda}_{kj} =
-\g{\Lambda}_{k+1,0} -
\sum_{i=1}^N \omega_i \left[ \m{y}_{3ki} + h \left(
\m{A}_{ki}\tr \g{\Lambda}_{ki}  + \m{Q}_{ki} \m{X}_{ki}
+ \m{S}_{ki} \m{U}_{ki} \right) \right] .
\]
By (\ref{Di0}), it follows that
\[
\g{\Lambda}_{k0} - \g{\Lambda}_{k+1,0}
- h \sum_{i=1}^N \omega_i (
\m{A}_{ki}\tr \g{\Lambda}_{ki}  + \m{Q}_{ki} \m{X}_{ki}
+ \m{S}_{ki} \m{U}_{ki}) + \m{y}_{4k} = \m{0} ,
\]
which is the 4th component of (\ref{linearproblem}).
This completes the proof.
\end{proof}
\section{Lipschitz continuity of $(\C{F} - \nabla \C{T}^*)^{-1}$ and
proof of the main theorem}
\label{Lip}
We begin by making the change of variables
$\m{X} = \m{Z} + \g{\chi}(\m{Y})$ where $\g{\chi}(\m{Y})$
denotes the solution of the state dynamics (\ref{h99})
corresponding to $\m{p}_{ki} = -\m{y}_{1ki}$ and $\m{q}_k = -\m{y}_{2k}$.
With this change of variables, $\m{y}_1$ and $\m{y}_2$ disappear
from the dynamics of the quadratic program (\ref{QP}) and the quadratic program
in $\m{Z}$ and $\m{U}$ reduces to
\begin{equation}\label{QPZ}
\left.
\begin{array}{cll}
\mbox{minimize}
&\frac{1}{2} \mathcal{Q}({\bf Z},{\bf U})
+ \bar{\C{L}}(\m{Z}, \m{U}, \m{Y}) &\\[.08in]
\mbox {subject to} &\sum_{j=0}^{N}{D}_{ij}{\bf Z}_{kj} =
h({\bf A}_{ki}{\bf Z}_{ki}+ {\bf B}_{ki}{\bf U}_{ki}) ,
& \m{U}_{ki} \in \C{U}, \\
&\m{Z}_{k0}=\m{Z}_{k-1,N}, \quad \m{Z}_{0N} = \m{0}, &
\end{array}
\right\}
\end{equation}
where $1\leq i \leq N$, $1 \le k \le K$, and
\begin{eqnarray}
\bar{\C{L}}(\m{Z}, \m{U}, \m{Y}) &=& \C{L}(\m{Z}, \m{U}, \m{Y}) +
\g{\chi}_{KN}(\m{Y})\tr \m{TZ}_{KN} \nonumber \\
&& + h \sum_{k=1}^K\sum_{i=1}^N \omega_i
\left[ \g{\chi}_{ki}\tr (\m{Y}) \m{Q}_{ki} \m{Z}_{ki}
+ \g{\chi}_{ki}\tr (\m{Y}) \m{S}_{ki}\m{U}_{ki} \right] .
\label{Lbar}
\end{eqnarray}

Note that $\m{Z}_{k0}$ can be eliminated from the optimization problem
with the substitution $\m{Z}_{k0}=\m{Z}_{k-1,N}$.
In the analysis that follows, it is assumed that the $\m{Z}_{k0}$
component of $\m{Z}$ has been deleted.
Note that $\C{Q}$ does not depend on $\m{Z}_{k0}$,
the $\m{Z}_{k0}$ in $L$ can be replaced by $\m{Z}_{k-1,N}$, and
the $\omega$-norm of $\m{Z}$ does not depend on $\m{Z}_{k0}$.
If $(\m{Z}^j, \m{U}^j)$ denotes the solution of (\ref{QPZ})
corresponding to $\m{Y}^j \in \C{Y}$, $j = 1$ and 2,
then by \cite[Lem.~4]{DontchevHager93}, the solution change
$\Delta\m{Z} = \m{Z}^1 - \m{Z}^2$ and $\Delta\m{U} = \m{U}^1 - \m{U}^2$
satisfies the relation
\begin{equation}\label{lemma4}
\C{Q}(\Delta\m{Z}, \Delta\m{U}) \le
|\bar{\C{L}}(\Delta\m{Z},\Delta\m{U}, \Delta \m{Y})|
\end{equation}
where $\Delta\m{Y} = \m{Y}^1 - \m{Y}^2$.

Observe that the quadratic $\C{Q}$ in (\ref{Q}) is expressed in terms
of the Hessian with respect to $\m{x}$ and $\m{u}$
of the Hamiltonian $H$ evaluated at
$(\m{x}^*(t_{ki}), \m{u}^*(t_{ki}), \g{\lambda}^*(t_{ki}))$;
by (A1), the Hessian of $H$ evaluated at
$(\m{x}^*(t), \m{u}^*(t), \g{\lambda}^*(t))$ for any $t \in [0, 1]$
has smallest eigenvalue greater than or equal to $\alpha > 0$.
It follows that
\[
\C{Q}(\Delta\m{Z}, \Delta\m{U}) \ge \alpha
\left( |\Delta \m{Z}_{KN}|^2 +
h \|\Delta\m{Z}\|_\omega^2 +
h \|\Delta\m{U}\|_\omega^2 \right) .
\]

Now consider the terms in $\C{L}$.
By the Schwarz inequality,
\[
\left| \sum_{k=1}^K
\sum_{i=1}^N \omega_i \Delta\m{y}_{3ki}\tr\Delta\m{Z}_{ki}\right| \le
\|\Delta \m{y}_3\|_\omega \|\Delta \m{Z}\|_\omega \le
\|\Delta \m{Y}\|_{\C{Y}} \|\Delta \m{Z}\|_\omega .
\]
Let $c$ denote a generic constant which is independent of $K$ and $N$.
For the control term in $\C{L}$, the triangle and Schwarz inequalities give
\begin{eqnarray*}
\left| \sum_{k=1}^K
\sum_{i=1}^N \omega_i \Delta\m{y}_{6ki}\tr\Delta\m{U}_{ki}\right| &\le&
c\|\Delta \m{y}_6\|_\infty
\sum_{k=1}^K \sum_{i=1}^N \omega_i |\Delta\m{U}_{ki}| \le
c\|\Delta \m{y}_6\|_{\infty} \sum_{k=1}^K \|\Delta\m{U}_k\|_\omega \\
&\le& ch^{-1/2}\|\Delta \m{y}_6\|_{\infty} \|\Delta\m{U}\|_\omega \le
c\|\Delta \m{Y}\|_{\C{Y}} \|\Delta\m{U}\|_\omega.
\end{eqnarray*}
The last inequality is due to the $h^{-1/2}$ factor in the $\C{Y}$-norm.

For the $\m{Z}_{k0}$-term in $\C{L}$, we have
\begin{equation}
\left| \Delta\m{Z}_{k0}\tr \left( \Delta\m{y}_{4k}
 + \sum_{i=1}^N \omega_i \Delta\m{y}_{3ki} \right) \right|
\le c \|\Delta \m{Z}\|_\infty
\left( |\Delta\m{y}_{4k}|
 + \sum_{i=1}^N \omega_i |\Delta\m{y}_{3ki}| \right) .
\label{Zk0}
\end{equation}
By Lemma~\ref{feasiblestate} with $\m{p}_{ki} = h\m{B}_{ki} \m{U}_{ki}$
and $\m{q}_k = 0$,
we have $\|\Delta \m{Z}\|_\infty \le ch^{1/2} \|\Delta \m{U}\|_\omega$.
Inserting this in (\ref{Zk0}) and applying the Schwarz inequality gives
\[
\left| \sum_{k=1}^{K}\Delta\m{Z}_{k0}\tr \left( \Delta\m{y}_{4k}
 + \sum_{i=1}^N \omega_i \Delta\m{y}_{3ki} \right) \right|
\le \|\Delta \m{U}\|_\omega
\left( |\Delta\m{y}_{4}|
 + \|\Delta\m{y}_{3}\|_\omega \right) \le c
\|\Delta \m{U}\|_\omega \|\Delta \m{Y}\|_{\C{Y}} .
\]

For the latter $\g{\chi}$ terms in (\ref{Lbar}), we have a bound such as
\begin{equation}\label{chi-bound0}
h \left| \sum_{k=1}^K\sum_{i=1}^N \omega_i
\g{\chi}_{ki}\tr (\Delta \m{Y}) \m{Q}_{ki} \Delta \m{Z}_{ki} \right| \le
ch \|\g{\chi}(\Delta \m{Y})\|_\omega \|\Delta \m{Z}\|_\omega.
\end{equation}
By Lemma~\ref{feasiblestate}, we have
\begin{equation}\label{chi-bound}
\|\g{\chi}(\Delta \m{Y})\|_\infty \le
ch^{-1/2} (\|\Delta \m{y}_1\|_\omega + |\Delta \m{y}_2|) \le
c h^{-1/2} \|\Delta \m{Y}\|_{\C{Y}}.
\end{equation}
Since $\|\g{\chi}(\Delta \m{Y})\|_\omega \le
h^{-1/2} \|\g{\chi}(\Delta \m{Y})\|_\infty$,
it follows from (\ref{chi-bound0}) and (\ref{chi-bound}) that
\[
h \left| \sum_{k=1}^K\sum_{i=1}^N \omega_i
\g{\chi}_{ki}\tr (\Delta \m{Y}) \m{Q}_{ki} \Delta \m{Z}_{ki} \right| \le
c \|\Delta \m{Y}\|_\omega \|\Delta \m{Z}\|_{\omega}.
\]

For the terminal term in (\ref{Lbar}), we have the bound
\[
|\g{\chi}_{KN}(\Delta\m{Y})\tr \m{T}\Delta\m{Z}_{KN}| \le
c |\g{\chi}_{KN} (\Delta\m{Y})| |\Delta\m{Z}_{KN}| \le
ch^{-1/2} (\|\Delta\m{y}_1\|_\omega + |\m{y}_2| ) |\Delta\m{Z}_{KN}|.
\]
The $\m{y}_5$ term in $\C{L}$ is similar.
By the Schwarz inequality,
\[
|\Delta \m{y}_5\tr \Delta \m{Z}_{KN}| \le |\Delta \m{y}_5 | |\Delta \m{Z}_{KN}|
\le h^{-1/2} \|\Delta \m{Y}\|_{\C{Y}} |\Delta \m{Z}_{KN}| ,
\]
where the $h^{-1/2}$ on the right cancels the $h^{1/2}$ factor
inside the $\C{Y}$-norm.

Combine these bounds for the linear term to obtain
\begin{eqnarray*}
|\bar{L}(\Delta\m{Z},\Delta\m{U}, \Delta \m{Y})| &\le&
c \|\Delta \m{Y}\|_{\C{Y}} \left(
h^{-1/2} |\Delta \m{Z}_{KN}| + \|\Delta\m{Z}\|_\omega + \|\Delta\m{U}\|_\omega
\right) \\
&=& ch^{-1/2} \|\Delta \m{Y}\|_{\C{Y}} \left( |\Delta \m{Z}_{KN}|
+ \sqrt{h}\|\Delta\m{Z}\|_\omega
+ \sqrt{h}\|\Delta\m{U}\|_\omega \right) \\
&\le& ch^{-1/2}  \|\Delta \m{Y}\|_{\C{Y}} \left(
|\Delta \m{Z}_{KN}|^2 + h \|\Delta\m{Z}\|_\omega^2 + h \|\Delta\m{U}\|_\omega^2
\right)^{1/2}.
\end{eqnarray*}
Combining the lower bound for $\C{Q}$ with the upper bound for
$\bar{\C{L}}$ gives
\begin{equation}\label{2norm}
\left(
|\Delta \m{Z}_{KN}|^2 + h \|\Delta\m{Z}\|_\omega^2 + h \|\Delta\m{U}\|_\omega^2
\right)^{1/2} \le ch^{-1/2}  \|\Delta \m{Y}\|_{\C{Y}} .
\end{equation}

Next, the $\omega$-type norm on the left side of (\ref{2norm}) will
be converted to an $\infty$-norm.
To do this, we first apply Lemma~\ref{feasiblestate} with
$\m{p}_{ki} = \Delta\m{y}_{1ki} + h \m{B}_{ki}\Delta \m{U}_{ki}$
and $\m{q}_k = \m{0}$.
The bound (\ref{2norm}) implies that
$\|\Delta \m{U}\|_\omega \le ch^{-1} \|\Delta \m{Y}\|_{\C{Y}}$;
consequently,
\begin{equation}\label{p-bound}
\|\m{p}\|_\omega \le
c \left( \|\Delta \m{y}_1\|_\omega + h \|\Delta \m{U}\|_\omega \right)
\le c \|\Delta \m{Y}\|_{\C{Y}} .
\end{equation}
It follows from (\ref{xjbound}) that
$\|\Delta \m{Z}\|_\infty \le ch^{-1/2} \|\Delta \m{Y}\|_{\C{Y}}$.
Hence, by (\ref{chi-bound}) we deduce that
\begin{equation}\label{X-bound}
\|\Delta \m{X}\|_\infty =
\|\Delta\m{Z} + \g{\chi}(\Delta\m{Y}) \|_\infty
\le ch^{-1/2} \|\Delta \m{Y}\|_{\C{Y}}.
\end{equation}
Since $\Delta \m{X}_{k0} = \Delta \m{X}_{k-1,0} + \Delta \m{y}_{2k}$,
it also follows that
\[
\|\Delta \m{X}_{k0}\|_\infty \le
\|\Delta \m{X}\|_\infty + |\Delta \m{y}_{2}| \le
ch^{-1/2} \|\Delta \m{Y}\|_{\C{Y}}.
\]

Let us now apply Lemma~\ref{feasiblecostate} with
\begin{eqnarray*}
\m{p}_{ki} &=& \Delta \m{y}_{3ki} +
h(\m{Q}_{ki} \Delta \m{X}_{ki} + \m{S}_{ki}\Delta\m{U}_{ki}), \quad
\Delta \g{\Lambda}_{K+1,0} = \m{T}\Delta\m{X}_{KN}, \quad \mbox{and}\\
\m{q}_k &=& \sum_{i=1}^N \omega_i \left( h \left[
\m{Q}_{ki} \Delta \m{X}_{ki} + \m{S}_{ki}\Delta\m{U}_{ki} \right]
+ \Delta \m{y}_{4ki} \right) .
\end{eqnarray*}
By (\ref{cjbound}), we have
\begin{equation}\label{lambda-bound}
\|\Delta \g{\Lambda}\|_{\infty} \le
c \left( \|\Delta \m{X}_{KN}\|_\infty + h^{-1/2} \|\m{p}\|_\omega
+ \sum_{k=1}^K |\m{q}_k|  \right) .
\end{equation}
By (\ref{X-bound}),
$\|\Delta \m{X}_{KN}\|_\infty \le ch^{-1/2}\|\Delta \m{Y}\|_{\C{Y}}$.
Exactly as in (\ref{p-bound}), $\|\m{p}\|_\omega \le c\|\Delta\m{Y}\|_{\C{Y}}$.
The Schwarz inequality yields
\begin{eqnarray*}
\sum_{k=1}^K |\m{q}_k| &\le&
c \sum_{k=1}^K \left( \|\Delta \m{y}_{2k}\|_\omega +
h \|\Delta \m{X}_k\|_\omega + h \|\Delta \m{U}_k\|_\omega \right) \\
&\le& c
h^{-1/2} \|\Delta \m{y}_2\|_\omega +
h^{1/2} \left[ \|\Delta \m{X}\|_\omega
+ \|\Delta \m{U}\|_\omega \right]
\le ch^{-1/2} \|\Delta \m{Y}\|_\omega .
\end{eqnarray*}
The last inequality utilizes both (\ref{2norm}) to bound the $\m{U}$ term
and (\ref{X-bound}) to bound the $\m{X}$ term.
Inserting these bounds in (\ref{lambda-bound}) yields
\begin{equation}\label{lambda-final}
\|\Delta \g{\Lambda}\|_\infty \le c h^{-1/2} \|\Delta \m{Y}\|_{\C{Y}} .
\end{equation}

Recall that $\m{R}_{ki} :=$
$\nabla_{uu}^2 H(\m{x}^*(t_{ki}), \m{u}^*(t_{ki}),\g{\lambda}^*(t_{ki}))$.
By (A1) the Hessian with respect to $\m{x}$ and $\m{u}$ of the Hamiltonian
$H$ evaluated at $(\m{x}^*(t), \m{u}^*(t), \g{\lambda}^*(t))$
for any $t \in [0, 1]$ has smallest eigenvalue greater than or equal to
$\alpha > 0$.
Consequently, the principal submatrix
$\m{R}_{ki}$ of the Hessian of the Hamiltonian
is positive definite with smallest eigenvalue
greater than or equal to $\alpha$.
It follows from the 6th component of the inclusion (\ref{linearproblem})
that the control associated with $\m{Y}$ solves the quadratic program
\[
\min_{\m{U}_{ki} \in \C{U}}
\;\; h \left( \frac{1}{2} \m{U}_{ki}\tr\m{R}_{ki} +
\m{X}_{ki}\tr \m{S}_{ki} + \g{\Lambda}_{ki}\tr \m{B}_{ki} \right)
\m{U}_{ki} + \m{y}_{6ki}\tr\m{U}_{ki} .
\]
Again by \cite[Lem.~4]{DontchevHager93}, the solution change associated
with the data change $\Delta \m{Y}$ has the bound
\[
h\alpha |\Delta \m{U}_{ki}|^2 \le \left|
h \left( \Delta \m{X}_{ki}\tr \m{S}_{ki} +
\Delta \g{\Lambda}_{ki}\tr \m{B}_{ki} \right) \Delta \m{U}_{ki} +
\Delta\m{y}_{6ki}\Delta \m{U}_{ki} \right|.
\]
Hence, we deduce that
\[
\|\Delta\m{U}_{ki}\|_\infty \le
|\Delta\m{U}_{ki}| \le c \left( \|\Delta \m{X}_{ki}\|_\infty +
\|\Delta \g{\Lambda}_{ki}\|_\infty + h^{-1}\|\Delta \m{y}_{6ki}\|_\infty
\right) .
\]
Utilizing the bounds (\ref{X-bound}) and (\ref{lambda-final}),
and the $h^{-1/2}$ factor associated with the 6-th
component of the $\C{Y}$-norm, yields
\begin{equation}\label{control-final}
\|\Delta\m{U}_{ki}\|_\infty \le c h^{-1/2} \|\Delta \m{Y}\|_{\C{Y}}.
\end{equation}
The bounds (\ref{X-bound}), (\ref{lambda-final}), and (\ref{control-final})
combine to establish the following Lipschitz continuity property:
\smallskip

\begin{lemma}\label{inf-bounds}
If {\rm (A1)}, {\rm (A2)}, and {\rm (P4)} hold,
then there exists a unique solution of $(\ref{linearproblem})$ for each
$\m{Y} \in \C{Y}$, and there exists a constant $c$,
independent of $K$ and $N$, such that the solution change
$\Delta \m{X}$, $\Delta \m{U}$, and $\Delta\g{\Lambda}$
relative to the change $\Delta \m{Y}$ satisfies
\[
\|(\Delta \m{X}, \Delta \m{U}, \Delta \g{\Lambda})\|_\infty
\le c h^{-1/2} \|\Delta \m{Y}\|_{\C{Y}}.
\]
\end{lemma}

Theorem~\ref{maintheorem} is proved using Proposition~\ref{prop}.
The Lipschitz constant $\gamma$ of Proposition~\ref{prop}
is given by $\gamma = ch^{-1/2}$
where $c$ is the constant of Lemma~\ref{inf-bounds}.
The terms involving
$\m{D}$, $\m{D}^\ddag$, $\g{\Lambda}_{k0}$, $\g{\Lambda}_{k+1,0}$,
$\m{X}_{k0}$, and $\m{X}_{k-1,N}$ are constants in the derivative
$\nabla \C{T}$ and hence these terms cancel when we compute
the difference $\nabla \C{T}(\g{\theta}) - \nabla \C{T}(\g{\theta}^*)$,
where $\g{\theta} = (\m{X}, \m{U}, \g{\Lambda})$ and
$\g{\theta}^* = (\m{X}^*, \m{U}^*, \g{\Lambda}^*)$.
We are left with terms involving the difference of
derivatives of $\m{f}$ or $C$ up to second order at points in a neighborhood
of $\g{\theta}^*$.
By the Smoothness assumption, these derivatives are Lipschitz continuous
in a neighborhood of $(\m{X}^*, \m{U}^*)$.
Hence, there exists constants $\tau$ and $r > 0$ such that
\begin{eqnarray*}
\|\nabla [\m{f}(\m{X}_{ki}, \m{U}_{ki}) -
\m{f}(\m{X}_{ki}^*, \m{U}_{ki}^*)]\|_\infty
&\le& \tau \|\g{\theta} - \g{\theta}^*\|_\infty, \\
\|\nabla [\nabla_x \m{H}(\m{X}_{ki}, \m{U}_{ki}, \g{\Lambda}_{ki}) -
\nabla_x \m{H}(\m{X}_{ki}^*, \m{U}_{ki}^*, \g{\Lambda}_{ki}^*) ]\|_\infty
&\le& \tau \|\g{\theta} - \g{\theta}^*\|_\infty, \\
\|\nabla [\nabla_u \m{H}(\m{X}_{ki}, \m{U}_{ki}, \g{\Lambda}_{ki}) -
\nabla_u \m{H}(\m{X}_{ki}^*, \m{U}_{ki}^*, \g{\Lambda}_{ki}^*) ]\|_\infty
&\le& \tau \|\g{\theta} - \g{\theta}^*\|_\infty, \\
\| \nabla [\nabla C(\m{X}_{KN}) - \nabla C(\m{X}_{KN}^*)]\|_\infty &\le&
\tau \|\g{\theta} - \g{\theta}^*\|_\infty,
\end{eqnarray*}
whenever $\|\g{\theta} -\g{\theta}^*\|_\infty \le r$.
In applying Proposition~\ref{prop}, we need a bound for the $\C{Y}$-norm
of $\nabla \C{T}(\g{\theta}) - \nabla \C{T}(\g{\theta}^*)$.
Taking into account the location of $h$'s in $\C{T}$ and the location
of $h$'s in the $\C{Y}$-norm, it follows from the Lipschitz bounds
relative to $\tau$ that there exists a constant $\kappa$ such that
\[
\|\nabla \C{T}(\g{\theta}) - \nabla \C{T}(\g{\theta}^*)\|_{\C{Y}} \le \kappa
h^{1/2}\|\g{\theta}- \g{\theta}^*\|_\infty,
\]
whenever $\|\g{\theta} -\g{\theta}^*\|_\infty \le r$.
Choose $r> 0$ smaller if necessary to ensure that $c \kappa r < 1$,
where $c$ is the constant in Lemma~\ref{inf-bounds}.
In Proposition~\ref{prop}, $\epsilon = \kappa h^{1/2} r$ and
$\gamma = ch^{-1/2}$.
Hence, $\gamma \epsilon = c\kappa r < 1$.
Referring to Lemma~\ref{residuallemma}, choose $N$ large enough or
$h$ small enough so that
\[
\mbox{dist}[\mathcal{T}(\g{\theta}^*), \C{F}(\m{U}^*)] \le
\frac {(1-\gamma \epsilon)r} {\gamma}.
\]
Combine Lemma~\ref{residuallemma} with (\ref{abs}) and the formula
$\gamma = ch^{-1/2}$ to obtain the bound (\ref{maineq})
of Theorem~\ref{maintheorem}.

The solution to $\C{T}(\m{X},\m{U}, \g{\Lambda}) \in \C{F}(\m{U})$
corresponds to the first-order optimality condition for either
(\ref{nlp}) or (\ref{D}).
We use the second-order sufficient optimality conditions to show that
this stationary point is a local minimum when it is sufficiently close to
$(\m{X}^*,\m{U}^*, \g{\Lambda}^*)$.
After replacing the KKT multipliers by the transformed quantities given by
$\g{\Lambda}_{ki}= \g{\lambda}_{ki}/\omega_i$,
the Hessian of the Lagrangian is a block
diagonal matrix with the following matrices forming the diagonal blocks:
\[
\begin{array}{ll}
\omega_i \nabla_{(x,u)}^2 H(\m{X}_{ki}, \m{U}_{ki}, 
\g{\Lambda}_{ki}), &
1 \le i < N, \\[.05in]
\omega_i \nabla_{(x,u)}^2 H(\m{X}_{ki}, \m{U}_{ki}, 
\g{\Lambda}_{ki}) +
\nabla_{(x,u)}^2 C(\m{X}_{{ki}}), & i = N,
\end{array}
\]
where $H$ is the Hamiltonian and $1 \le k \le K$.
In forming the Hessian with respect to $\m{X}$ and $\m{U}$,
the variables are arranged as follows:
$\m{X}_{k1}$, $\m{U}_{k1}$, $\m{X}_{k2}$, $\m{U}_{k2}$,
$\ldots$, $\m{X}_{kN}$, $\m{U}_{kN}$, $1 \le k \le K$.
By (A1) the Hessian is positive definite when evaluated at
$(\m{X}^*, \m{U}^*, \g{\Lambda}^*)$.
Since the second derivatives of $C$ and $\m{f}$ are
Lipschitz continuous and the iterates converge to
$(\m{X}^*, \m{U}^*, \g{\Lambda}^*)$ in the sup-norm
by Theorem~\ref{maintheorem},
the Hessian of the Lagrangian evaluated at the discrete iterates
is positive definite for $N$ sufficiently large
or for $h$ sufficiently small with $N \ge 2$.
Hence, by the second-order sufficient optimality condition
\cite[Thm. 12.6]{NocedalWright2006}, the discrete state and control
is a strict local minimizer of (\ref{nlp}).
This completes the proof of Theorem~\ref{maintheorem}.
\section{Numerical illustrations}
\label{numerical}
In this section we analyze the errors associated with the proposed
Radau $hp$-collocation method using
numerical examples with known analytic solutions.
Consequently, it is possible to precisely determine the
error in the $hp$-approximations.
More complex examples, which do not have known analytic solutions, appear in
both \cite{Patterson2015} and at the GPOPS-II examples website:
\begin{center}
http://www.gpops2.com/Examples/Examples.html
\end{center}
In \cite{Patterson2015} it is observed that the solutions computed by
Radau $hp$-collocation are in close agreement to the solutions computed by
Betts' Sparse Optimization Suite (SOS) \cite{Betts2013}.
\subsection{Example 1}
First we consider the unconstrained control problem given by
\begin{align}
\mbox{min}\left\{-x(2): \;\dot{x}(t)=\frac{5}{2}(-x(t)+x(t)u(t)-u^2(t)),
\; x(0)=1\right\}.
\end{align}
The optimal solution and associated costate are
\begin{eqnarray*}
x^*(t) &=& 4/a(t), \quad a(t) = 1 + 3 \exp (2.5t), \\
u^*(t) &=& x^*(t)/2, \\
\lambda^* (t) &=& -a^2(t) \exp (-2.5t)/[\exp (-5) + 9 \exp(5) + 6].
\end{eqnarray*}
The time domain [0,2] is divided into equally spaced mesh intervals,
and on each mesh interval, we collocate at the Radau points
using polynomials of the same degree.
We consider polynomials of degree $N=2$, 3, and 4.
Convergence to the true solution is achieved by increasing the
number of mesh intervals.
Figure~\ref{example1} plots the base 10 logarithm of the error
at the collocation points in the sup-norm  versus the base 10 logarithm
of mesh size.
The results were obtained
using the software GPOPS-II \cite{GPOPS2} and the optimizer
IPOPT \cite{Biegler08} to solve the discrete nonlinear program.
The markers plotted in Figure~\ref{example1} correspond to the sup-norm
error at a given value for $h$, while the lines have slope $N+2$ for the
state and control, and $N+1$ for the costate.
The vertical placement of each line yields the least squares fit to the markers.
Observe that the error decays roughly linearly in this log-log plot,
and the pointwise error is roughly $O(h^{N+2})$
in the state and control, and $O(h^{N+1})$ in the costate for fixed $N$.
\begin{figure}
\begin{center}
\includegraphics[scale=.45]{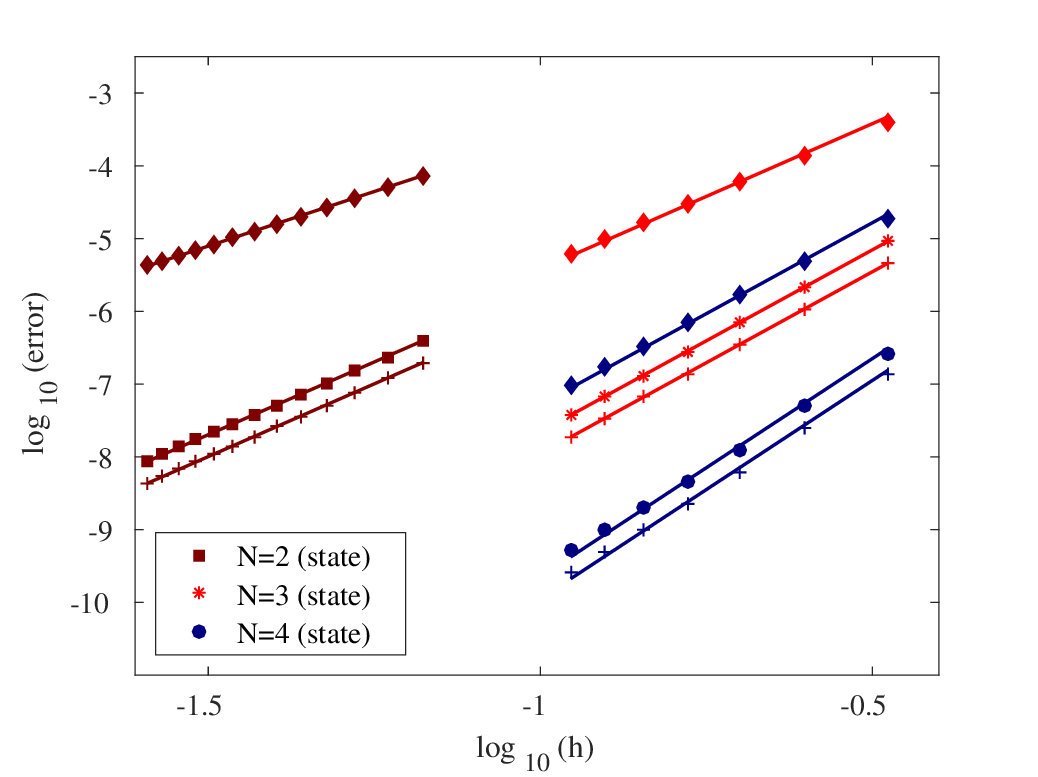}
\caption{The logarithm of the sup-norm error in Example~1
as a function of mesh size for polynomials of degree $N=2$, 3, and 4.
The errors in the controls, marked by plus signs, are beneath
the state error plots.
The errors in the costate, marked by diamonds, are above the state error plots.
\label{example1}}
\end{center}
\end{figure}

The bound given in Theorem~\ref{maintheorem} for fixed $N$ is $O(h^{N-1})$,
which is much slower than the observed convergence rate $O(h^{N+1})$.
This discrepancy could be due to either the simple nature
of the example, or to looseness in the analysis.
In our analysis, the exponent of $h$ is reduced by the following effects:
\begin{itemize}
\item[(a)]
Although the state is approximated by a polynomial of degree $N$,
the costate is approximated by a polynomial of degree $N-1$.
This difference between the state and the costate becomes apparent in
Proposition~\ref{equiv}.
We are not free to choose the costate polynomial, its degree comes
from the KKT conditions.
In the analysis of the residual given in Lemma~\ref{residuallemma},
the reduced degree for the costate polynomial implies that the
exponent of $h$ in the bound (\ref{delta}) is the minimum of $N$ and
$\eta$ rather than the minimum of $N+1$ and $\eta$.
\item[(b)]
In our analysis at the end of Section~\ref{Lip}, we showed that by taking
$r$ small enough, the expression $\gamma \epsilon$ in the denominator
of (\ref{abs}) was strictly bounded from one.
The analysis also showed that that the Lipschitz constant
satisfied $\gamma \le ch^{-1/2}$.
Hence, we lose a half power of $h$ through the Lipschitz constant in the
error bound (\ref{abs}).
\end{itemize}
If example 1 indeed represents the typical behavior of the error,
then the analysis must be sharpened to address the losses described in
(a) and (b).

It is interesting to compare the analysis in this paper to the analysis
of Runge-Kutta schemes given in \cite{BonnansVarin06,Hager99c}.
For a fixed $N$, the Radau scheme in this paper is equivalent to a
Runge-Kutta scheme where the $\m{A}$ matrix and $\m{b}$
vector of \cite{Hager99c} describing the Runge-Kutta scheme are
$\m{D}_{1:N}^{-1}/2$ and the last row of $\m{D}_{1:N}^{-1}/2$ respectively.
For $N = 2$ and $N = 3$, the corresponding Runge-Kutta schemes have order
3 and 4 respectively, which means that the error in the Runge-Kutta schemes
are $O(h^3)$ and $O(h^4)$ respectively.
This exactly matches the costate error for the $hp$-scheme in this example.
A fundamental difference between the results of \cite{Hager99c} and the
results in this paper is that \cite{Hager99c} estimates the error at
the mesh points, and there is no information about the error at the
intermediate points, while in Theorem~\ref{maintheorem}, we estimate
the error at both collocation and mesh points.
In the $hp$-framework, it is important to have estimates at the collocation
points since $K$ could be fixed, and the convergence is achieved by letting
$N$ grow.

Based on the theory developed in the paper \cite{BonnansVarin06}
of Bonnans and Laurent-Varin, many conditions must be satisfied to
achieve high order convergence of a Runge-Kutta scheme for optimal control
(4116 conditions for order 7).
Potentially, the $hp$-scheme based on Radau collocation could be used to
generate high order Runge-Kutta schemes.

Next, we examine in Figure~\ref{exampleSingleInterval}
the exponential convergence rate predicted by
Theorem~\ref{maintheorem} when there is a single interval and
the degree of the polynomials is increased.
Since the plot of the base-10 logarithm of the error versus the degree of of the
polynomial is nearly linear, the error behaves like
$c 10^{-\alpha N}$ where $\alpha \approx 0.6$ for either the state
or the control and $\alpha \approx 0.8$ for the costate.
Since the solution to this problem is infinitely smooth, we can take
$\eta = N$ in Theorem~\ref{maintheorem}.
The error bound in Theorem~\ref{maintheorem} is somewhat complex since it
involves the derivatives of the solution.
Nonetheless, when we take the base-10 logarithm of the error bound,
the asymptotically dominant term appears to be $-N\log_{10} N$ for
Example~1.
Consequently, the slope of the curve in the error bound varies like
$-\log_{10} N$.
For $N$ between 4 and 16, $\log_{10} N$ varies from about 0.6 to 1.2.
Hence, our observed slopes 0.6 and 0.8 fall in the anticipated range.
\begin{figure}
\begin{center}
\includegraphics[scale=.5]{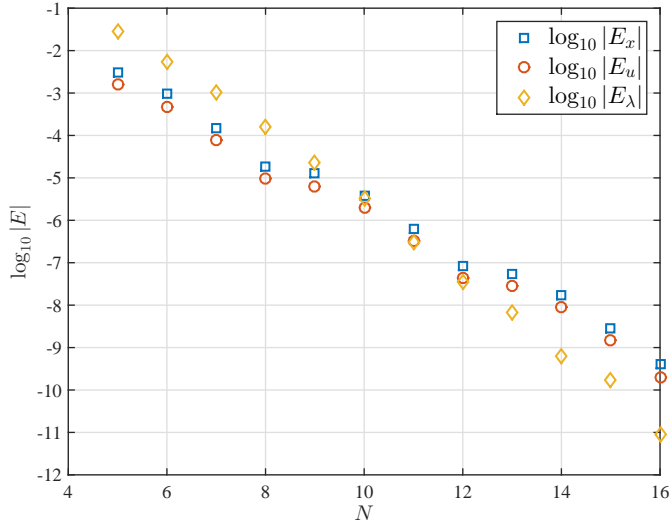}
\caption{The base 10 logarithm of the error in the sup-norm as a function of
the number of collocation points for Example~1.
\label{exampleSingleInterval}}
\end{center}
\end{figure}

\subsection{Example 2}
Next we consider the problem \cite{Hager84b} given by
\begin{eqnarray*}
&\mbox{minimize}&\quad \frac{1}{2}\int_{0}^{1}[x^2(t)+u^2(t)]\;dt\\
&\mbox{subject to}& \quad \dot{x}(t)=u(t),\quad u(t)\leq 1, \quad
x(t)\leq \frac{2\sqrt{e}}{1-e}\quad \mbox{for all }
t\in [0,1], \\
&& \quad x(0)=\frac{5e+3}{4(1-e)}.
\end{eqnarray*}
The exact solution to this problem is
\[
\begin{array}{lll}
{0\leq t\leq \frac{1}{4}}:
& x^*(t)=t-\frac{1}{4}+\frac{1+e}{1-e},
& u^*(t)=1,\\[.05in]
{\frac{1}{4} \leq t\leq \frac{3}{4}}:
& x^*(t)=\frac{e^{t-\frac{1}{4}}}{1-e}(1+e^{\frac{3}{2}-2t}),
& u^*(t)=\frac{e^{t-\frac{1}{4}}}{1-e}(1-e^{\frac{3}{2}-2t}),\\[.05in]
{\frac{3}{4}\leq t\leq 1}:
& x^*(t)=\frac{2\sqrt{e}}{1-e},
& u^*(t)=0.
\end{array}
\]
The solution of this problem is smooth on the three intervals
$[0, 0.25]$, $[0.25, 0.75]$, and $[0.75, 1.0]$, however, at the contact
points where one of the constraints changes from active to inactive,
there is a discontinuity in the derivative of the optimal control and
a discontinuity in the second derivative of the optimal state.
The goal with this test problem is to determine whether exponential
convergence occurs for the $hp$-scheme with a careful choice of the mesh,
and whether a state constrained problem, which is not covered by the
error analysis in this paper, possesses similar errors bounds to those
for control constrained problems.

First, we solve the problem using $K = 1$, in which case
convergence is achieved by increasing the degree $N$ of the polynomials.
In Figure~\ref{fig2}(a) we plot the logarithm of the error at the collocation
points in the sup-norm versus the logarithm of the polynomial degree.
Convergence occurs, but it is slow due to the discontinuity in the derivatives.
\begin{figure}
(a)\includegraphics[scale=.35]{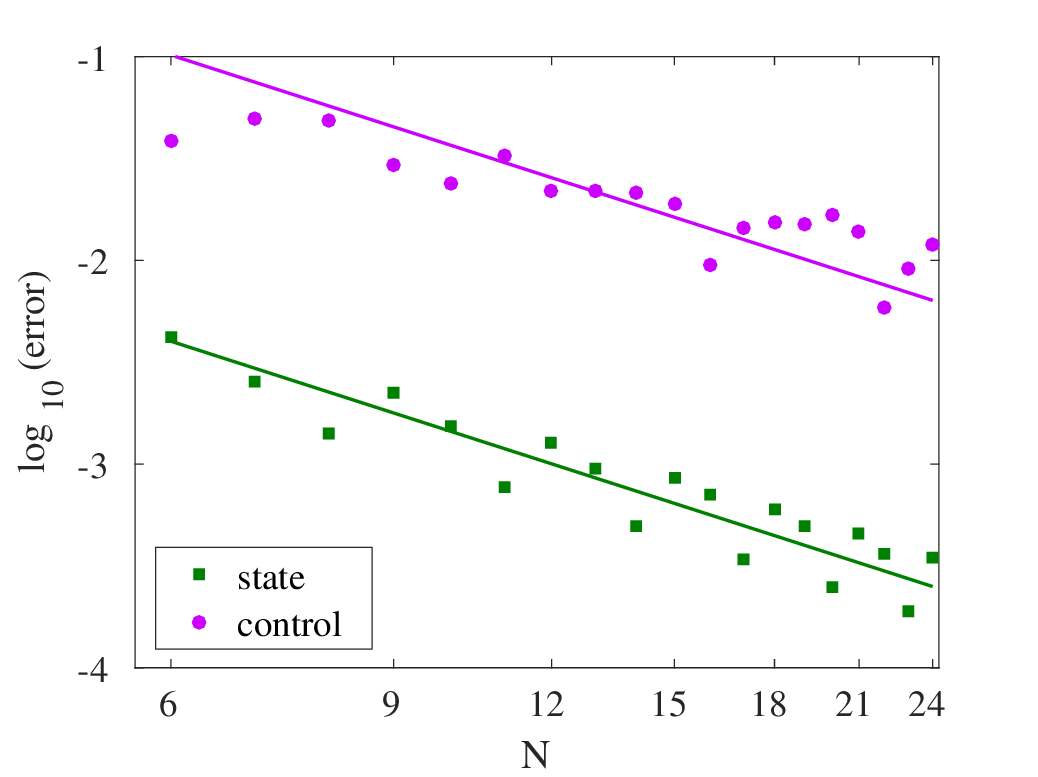}
(b)\includegraphics[scale=.35]{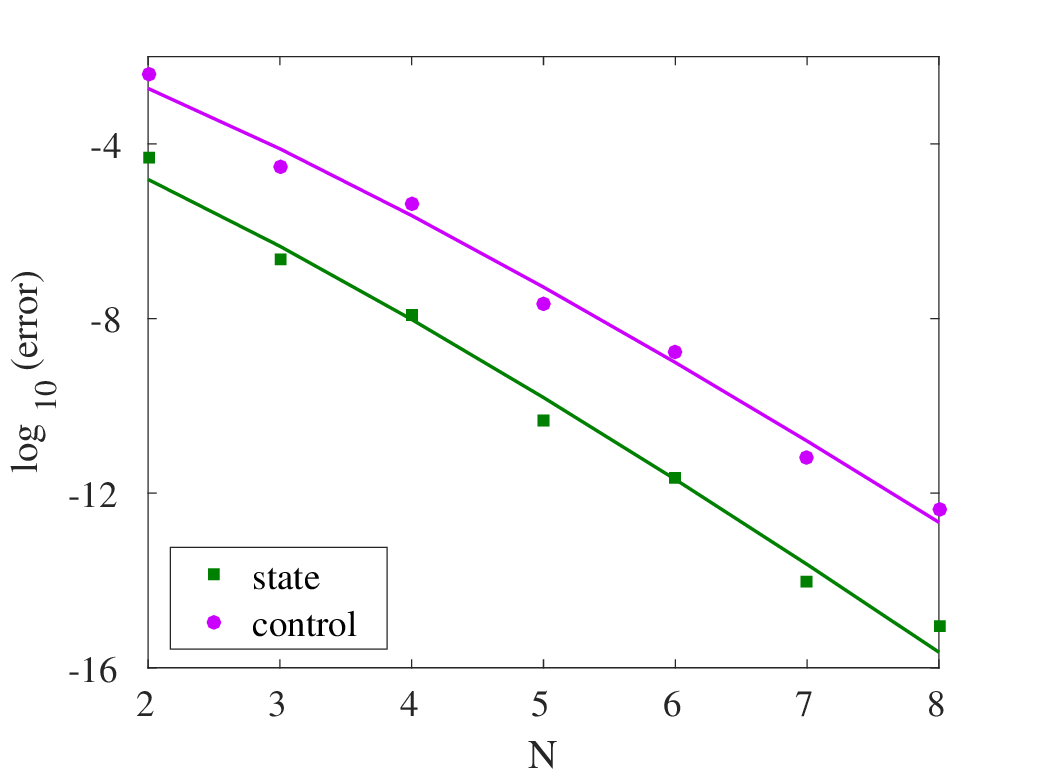}
\caption{The error in the solution to Example~2 as a function of the degree
of the polynomials used in the $hp$-approximation.
In (a) the polynomials are defined on the interval $[0, 1]$.
In (b) there are three mesh intervals
$[0, 0.25]$, $[0.25, 0.75]$, and $[0.75, 1.0]$,
and different polynomials of the same degree are used on each mesh interval.
\label{fig2}
}
\end{figure}
The lines in Figure~\ref{fig2}(a) have slope $-2$; their vertical placement
was chosen to achieve the best least squares fit to the markers
(the measured error).
Since the logarithm of the error is approximately fit by a line of slope
$-2$, the error decays like $c/N^2$, which is faster than what might
be expected from a bound like that given in Theorem~\ref{maintheorem}
with regularity $\C{H}^{2.5-\epsilon}$ for any $\epsilon > 0$.

Next, we divide the time interval [0,1] into three subintervals
$[0, 0.25]$, $[0.25, 0.75]$, and $[0.75, 1.0]$,
and use different polynomials of the same degree on each subinterval.
By this careful choice of the mesh intervals,
we obtain an exponential convergence rate in Figure \ref{fig2}(b).
Comparing Figures~\ref{fig2}(a) and \ref{fig2}(b),
we see that a huge improvement in the error is possible when we have 
good estimates for the contact points where the constraints change
between active and inactive.
Note that 16-digit accuracy was obtained in Figure~\ref{fig2}(b) by
using MATLAB's {\sc quadprog} to solve the quadratic program associated
with the $hp$-discretization of Example~2.

In a very rough sense, the error bound given by Theorem~\ref{maintheorem}
for a smooth problem has the general form $c_1(c_2/N)^N$.
The continuous curves plotted in Figure~\ref{fig2}(b) were obtained by choosing
$c_1$ and $c_2$ to achieve the least squares best fit to the markers
(the measured error).
For the state variable, $(c_1, c_2) = (0.0016, 0.1990)$, while for the
control $(c_1, c_2) = (0.0950, 0.2801)$.
Hence, it seems plausible that a state-constrained control problem may
possess an error bound similar to that established in Theorem~\ref{maintheorem}
for control constrained problems.

\section{Conclusions}
A convergence rate is derived for an
$hp$-orthogonal collocation method based on the Radau quadrature points
applied to a control problem with convex control constraints.
If the problem has a smooth local solution and a Hamiltonian
which satisfies a strong convexity assumption,
then the discrete approximation has a
local minimizer in a neighborhood of the continuous solution.
For the $hp$-scheme, both the number of mesh intervals in
the discretization and the degree of the polynomials on each mesh interval
can be freely chosen.
As the number of mesh intervals increases, convergence occurs at a
polynomial rate relative to the mesh width.
When there is control over the growth in derivatives,
the convergence rate is exponentially fast relative to the polynomial degree.
Convergence rates were investigated further using numerical examples.
When the polynomial degree is fixed and the mesh width tends to zero,
the observed convergence rate was faster
than the rate associated with the error bound.
For a problem with control and state constraints,
exponentially fast convergence was observed when
mesh points are located at the contact points
where the constraints change between active and inactive.
Based on the numerical results, it seems plausible that the
convergence result established for control constrained problem could
extend to problems with state constraints.
\section{Appendix 1: Proof of (P1) and (P2)}
\label{appendix1}
We analyze (P1) and (P2) when $\tau_i$, $1 \le i \le N$, are either
the Radau quadrature points analyzed in this paper, or the Gauss quadrature
points studied in \cite{HagerMohapatraRao16b}.
\smallskip

\begin{lemma}\label{P1P2}
For either the Gauss or Radau quadrature points,
the rows of the matrix $[\m{W}^{1/2}\m{D}_{1:N}]^{-1}$ have Euclidean
length bounded by $\sqrt{2}$.
For the Gauss quadrature points, $\|\m{D}_{1:N}^{-1}\|_\infty \le 2$,
and $\|\m{D}_{1:N}^{-1}\|_\infty$ approaches $2$ as $N$ tends to infinity,
while for the Radau quadrature points,
$\|\m{D}_{1:N}^{-1}\|_\infty = 2$.
\end{lemma}
\smallskip

\begin{proof}
Given $\m{p} \in \mathbb{R}^N$, let $p \in \C{P}_N$ denote the polynomial that
satisfies $p(-1) = 0$ and $p(\tau_i) = p_i$, $1 \le i \le N$.
Let $\dot{\m{p}} \in \mathbb{R}^N$ denote the vector with components
$\dot{p}_i = \dot{p}(\tau_i)$, and let $\ell_j$ be the Lagrange polynomial
defined by
\[
\ell_j(\tau) = 
\prod^{N}_{\substack{i=1\\i\neq j}}
\frac{\tau-\tau_i}{\tau_j-\tau_i}, \quad 1 \le j \le N.
\]
The identity
\begin{equation}\label{dot}
\dot{p}(\tau) = \sum_{j=1}^N \ell_j(\tau) \dot{p}_j
\end{equation}
holds since $\dot{p} \in \C{P}_{N-1}$ and the polynomials on each side of
(\ref{dot}) are equal at the $N$ quadrature points.
Integrate (\ref{dot}) to obtain
\begin{equation}\label{integrate}
p_i = \int_{-1}^{\tau_i} \dot{p}(\tau)\; d\tau =
\sum_{j=1}^N \left( \int_{-1}^{\tau_i} \ell_j(\tau) \; d\tau \right)
\dot{p}_j .
\end{equation}
Since $\m{D}$ is a differentiation matrix and $p(-1) = 0$, it follows that
$\m{D}_{1:N}\m{p} = \dot{\m{p}}$.
If the vector $\dot{\m{p}} = \m{0}$, then the polynomial $\dot{p} = 0$
since $\dot{p}$ has degree $N-1$
and vanishes at $N$ points.
Since $p(-1) = 0$, it follows that polynomial $p = 0$,
which implies that the vector $\m{p} = \m{0}$.
Hence, $\m{D}_{1:N}$ is invertible, and
$\m{p} = \m{D}^{-1}\dot{\m{p}}$.
Comparing the equality $\m{p} = \m{D}^{-1}\dot{\m{p}}$ to (\ref{integrate}),
we deduce that
\begin{equation}\label{Dinverse}
(\m{D}^{-1})_{ij} = \int_{-1}^{\tau_i} \ell_j(\tau)\; d\tau.
\end{equation}

Choose any $s \in [-1, 1]$ and define
\[
d_j(s) = \int_{-1}^s \ell_j(\tau) \; d\tau \quad \mbox{and} \quad
R(s) = \sum_{j=1}^N \frac{d_j(s)^2}{\omega_j}.
\]
Observe that $(\m{D}^{-1})_{ij} = d_j(\tau_i)$ and $R(\tau_i)$ is the
square of the Euclidean length of row $i$ in $(\m{W}^{1/2}\m{D})^{-1}$.
Let $q \in \C{P}_{N-1}$ be the polynomial defined by
\[
q(\tau) = \sum_{j=1}^N \frac{d_j(s) \ell_j(\tau)}{\omega_j}.
\]
Hence, by the triangle and Schwarz inequalities,
\begin{equation}\label{R1}
R(s) = \int_{-1}^s q(\tau) \; d\tau \le \int_{-1}^1 |q(\tau)| \;d \tau
\le \sqrt{2} \left( \int_{-1}^1 q(\tau)^2 \; d \tau \right)^{1/2}.
\end{equation}
Since $q^2 \in \C{P}_{2N-2}$, both Radau and Gauss quadrature are exact, and
\begin{equation}\label{R2}
\int_{-1}^1 q(\tau)^2 \; d\tau = \sum_{i=1}^N \omega_i q(\tau_i)^2,
\end{equation}
where the $\tau_j$ are either the Radau or Gauss quadrature points and
the $\omega_j$ are the associated weights.
Since $\ell_j(\tau_i) = 1$ for $i = j$ and $\ell_j(\tau_i) = 0$ otherwise,
it follows from the definition of $q$ that $q(\tau_i) = d_i(s)/\omega_i$.
This substitution in (\ref{R2}) yields
\begin{equation}\label{R3}
\int_{-1}^1 q(\tau)^2 \; d\tau =
\sum_{i=1}^N \frac{ d_i (s)^2}{\omega_j} = R(s) .
\end{equation}
Equating the expressions (\ref{R1}) and (\ref{R3}) implies that
\[
\left( \int_{-1}^1 q(\tau)^2 \; d \tau \right)^{1/2} \le \sqrt {2}.
\]
By (\ref{R3}), $R(s) \le 2$ for any $s \in [-1, 1]$.
In particular, $R(\tau_i) \le 2$ for $1 \le i \le N$.
Since $R(\tau_i)$ is the square of the Euclidean length of row $i$
in $(\m{W}^{1/2}\m{D})^{-1}$, the rows of $(\m{W}^{1/2}\m{D})^{-1}$
have Euclidean length bounded by $\sqrt{2}$.
This result holds for both the Radau and Gauss quadrature points since
since $q^2 \in \C{P}_{2N-2}$, and both Radau and Gauss quadrature are exact
for polynomials of this degree.

If $\m{r}$ is a row of $\m{D}_{1:N}^{-1}$, then by the
Schwarz inequality and the fact that the quadrature weights sum to 2 and
the rows of the matrix $[\m{W}^{1/2}\m{D}_{1:N}]^{-1}$ have Euclidean
length bounded by $\sqrt{2}$, we have
\begin{equation}\label{Euclidean}
\sum_{i=1}^N |r_i| =
\sum_{i=1}^N \sqrt{\omega_i} \left( |r_i|/\sqrt{\omega_i}\right) \le
\left( \sum_{i=1}^N \omega_i \right)^{1/2}
\left( \sum_{i=1}^N r_i^2/\omega_i \right)^{1/2} \le 2.
\end{equation}
Consequently, the absolute row sums for $\m{D}_{1:N}^{-1}$
are all bounded by 2, or equivalently, $\|\m{D}_{1:N}^{-1}\|_\infty \le 2$.
Given any polynomial $p \in \C{P}_N$ with $p(-1) = 0$ and
$|\dot{p}(\tau_i)| \le 1$ for $1 \le i \le N$, it is observed in
Section~9 of \cite{HagerMohapatraRao16b} that
$\|\m{D}_{1:N}^{-1}\|_\infty \ge$ $\max \{p(\tau_i) : 1 \le i \le N\}$.
Take $p(\tau) = 1 + \tau$ to deduce
that $\|\m{D}_{1:N}^{-1}\|_\infty \ge 1 + \tau_N$.
Hence, $1 + \tau_N \le \|\m{D}_{1:N}^{-1}\|_\infty \le 2$.
Since $\tau_N = 1$ for the Radau points, it follows that
$\|\m{D}_{1:N}^{-1}\|_\infty = 2$.
For the Gauss points, $\tau_N$ approaches 1 as $N$ tends to infinity;
consequently, $\|\m{D}_{1:N}^{-1}\|_\infty$ approaches 2 as $N$ tends to
infinity for the Gauss points.
\end{proof}
\section{Appendix 2: An analytic formula for $({\bf D}^\ddag)^{-1}$}
\label{appendix2}
Before stating property (P3) in the Introduction,
we showed that ${\bf D}^\ddag$ is an invertible matrix.
In this section, we give an analytic formula for the inverse.
\begin{proposition}\label{deriv_exact}
The inverse of ${\bf D}^\ddag$ is given by
%
\[
\begin{array}{llll}
D^{\ddag \; -1}_{ij} &=& \omega_N M_j(1) +
\displaystyle{\int_1^{\tau_i} M_j (\tau) d\tau,} &
1 \le i < N, \;\; 1 \le j < N, \\[.05in]
D^{\ddag \; -1}_{iN} &=& -\omega_N , & 1 \le i \le N, \\[.05in]
D^{\ddag \; -1}_{Nj} &=& \omega_N M_j (1) , & 1 \le j < N,
\end{array}
\]
where $M_j$, $1 \le j < N$, is the Lagrange interpolating basis
relative to the point set $\tau_1$, $\ldots$, $\tau_{N-1}$.
That is,
\[
M_j (\tau) =
\displaystyle\prod_{\substack{i=1\\ i\neq j}}^{N-1}
\frac{\tau-\tau_i}{\tau_j-\tau_i},\quad j=1,\ldots, N-1.
\]
\end{proposition}
\smallskip
\begin{proof}
The relation (\ref{h282}) holds for any polynomial $p$ of degree at
most $N-1$.
Let $\dot{\m{p}} \in \mathbb{R}^N$ denote the vector with
$i$-th component $\dot{p}(\tau_i)$.
In vector form, the system of equations (\ref{h282}) can be expressed
${\bf D}^\ddag{\bf p}=$ $\dot{\bf p} - \m{e}_N p(1)/\omega_N$.
Multiply by
$\m{D}^{\ddag \; -1}$ and exploit the identity
$\m{D}^{\ddag \; -1} \m{e}_N = -\omega_N \m{1}$ of (\ref{h999})
to obtain
\begin{equation}\label{h998}
{\bf D}^{\ddag \; -1} \dot{\m{p}} = \m{p} - \m{1}p(1) .
\end{equation}

Since $\dot{p}$ is a polynomial of degree at most $N-2$,
we can only specify the derivative of $p$ at $N-1$ distinct points.
Given any $j$ satisfying $1 \le j < N$, let us insert in (\ref{h998})
a polynomial $p \in \C{P}_{N-1}$ satisfying
\[
\dot{p} (\tau_j) = 1 \quad \mbox{and} \quad
\dot{p}(\tau_i) = 0
\mbox{ for all } i < N, \; \; i \ne j.
\]
A specific polynomial with this property is
\begin{equation}\label{h996}
p(\tau) = \int_{1}^{\tau} M_j(\tau) d\tau.
\end{equation}
Since $p_N = p(1) = 0$, the last component of the right side of
(\ref{h998}) vanishes to give the relation
$D_{Nj}^{\ddag \; -1} + D_{NN}^{\ddag \; -1} \dot{p}(1) = 0$.
In (\ref{h999}) we showed that all the elements in the last column of
$\m{D}^{\ddag \; -1}$ are equal to $-\omega_N$, and
by (\ref{h996}), $\dot{p}(1) = M_j (1)$.
Hence, we obtain the relation
\begin{equation}\label{h997}
D_{Nj}^{\ddag \; -1} =
-D_{NN}^{\ddag \; -1} \dot{p}(1) =
\omega_N \dot{p}(1) =
\omega_N M_j (1), \;\; 1 \le j < N.
\end{equation}

Finally, let us consider $D_{ij}^{\ddag \; -1}$ for $i < N$ and $j < N$.
We combine the $i$-th component of (\ref{h998}) for $i < N$ with
(\ref{h996}) to obtain
\begin{equation}\label{h995}
({\bf D}^{\ddag \; -1} \dot{\m{p}})_i = \int_{1}^{\tau_i} M_j (\tau) d\tau .
\end{equation}
Recall that all components of $\dot{\m{p}}$ vanish except for the $j$-th,
which is 1, and the $N$-th, which is $M_j(1)$ by (\ref{h996}).
Hence, (\ref{h995}) and the fact that
the elements in the last column of
$\m{D}^{\ddag \; -1}$ are all $-\omega_N$ yield
\[
D_{ij}^{\ddag \; -1} = \int_{1}^{\tau_i} M_j (\tau) d\tau -
D_{iN}^{\ddag \; -1} M_j(1) =
\omega_N M_j (1) + \int_{1}^{\tau_i} M_j (\tau) d\tau
\]
This completes the proof.
\end{proof}

Tables~\ref{P3} and \ref{P4} show $\|\m{D}^{\ddag \; -1}\|_\infty$
and the maximum Euclidean norm of the rows of
$\m{D}^{\ddag \; -1}\m{W}^{-1/2}$ for an increasing sequence of dimensions.
The norms in Table~\ref{P3} approach 2 as $N$ grows, consistent with (P3),
while the norms in Table~\ref{P4} approach $\sqrt{2}$, consistent with (P4).
\begin{table}[ht]
\centering
\begin{tabular}{c|c|c|c|c|c|c}
\hline\hline
$N$ & 25 & 50 & 75 & 100 & 125 & 150 \\
\hline
norm & 1.995376 & 1.998844 & 1.999486 & 1.999711 & 1.999815 & 1.999871  \\
\hline\hline
$N$ & 175 & 200 & 225 & 250 & 275 & 300 \\
\hline
norm & 1.999906 & 1.999928 & 1.999943 & 1.999954 & 1.999962 & 1.999968\\
\hline
\end{tabular}
\vspace*{.1in}
\caption{$\|{\bf D}^{\ddag\; -1}\|_\infty$
\label{P3}}
\end{table}
\begin{table}[ht]
\centering
\begin{tabular}{c|c|c|c|c|c|c}
\hline\hline
$N$ & 25 & 50 & 75 & 100 & 125 & 150 \\
\hline
norm & 1.412209 & 1.413691 & 1.413982 & 1.414083 & 1.414130 & 1.414156 \\
\hline\hline
$N$ & 175 & 200 & 225 & 250 & 275 & 300 \\
\hline
norm & 1.414171 & 1.414181 & 1.414188 & 1.414193 & 1.414196 & 1.414199 \\
\hline
\end{tabular}
\vspace*{.1in}
\caption{Maximum Euclidean norm for the rows of
$[\m{W}^{1/2}{\bf D}^\ddag]^{-1}$
\label{P4}}
\end{table}


\bibliographystyle{siam}

\end{document}